\documentclass{amsart}

\usepackage{amsfonts}
\usepackage{amsmath}
\usepackage{amssymb}
\usepackage{amsthm}
\usepackage{stmaryrd}

\makeatletter
\renewcommand\@seccntformat[1]{\csname the#1\endcsname.\quad}

\newtheorem{theorem}{Theorem}[section]

\newtheorem{corollary}[theorem]{Corollary}

\newtheorem{definition}[theorem]{Definition}
\newtheorem{example}[theorem]{Example}

\newtheorem{lemma}[theorem]{Lemma}

\newtheorem{proposition}[theorem]{Proposition}
\newtheorem{remark}[theorem]{Remark}

\newcommand{\norm}[3]{\ensuremath{\left\Vert#1\right\Vert_{#2}^{#3}}}
\newcommand{\abs}[3]{\ensuremath{\left\vert#1\right\vert_{#2}^{#3}}}
\DeclareMathOperator{\supp}{supp}

\begin{document}

\author[M. Ciesielski \& G. Lewicki]{Maciej Ciesielski$^{1*}$ and Grzegorz Lewicki}

\title[On a certain class of norms in semimodular spaces]{On a certain class of norms in semimodular spaces and their monotonicity properties}

\begin{abstract}
 Let $ X$ be a linear space over $\mathbb{K},$ $ \mathbb{K} = \mathbb{R}$ or   $ \mathbb{K} = \mathbb{C}$ and let for $ n \geq 2$ $\rho_i$ be $s$-convex semimodular defined on $X$ for any $i\in\{1,\dots,n-1\}$. Put $ \rho=\max_{1\leq{i}\leq{n-1}}\{\rho_i\}$ and 
 $ X_{\rho}= \{ x \in X: \rho(dx) < \infty \hbox{ for some } d > 0 \}.$ In this paper we define a new class of $s$-norms (norms if $s=1$) on $ X_{\rho}.$ In particular, our defintion generalizes in a natural way the Orlicz-Amemiya and Luxemburg norms defined for $s$-convex semimodulars. Then, we investigate order continuous, the Fatou Property and various monotonicity properties of semimodular spaces equipped with these $s$-norms. 
\end{abstract}

\maketitle

\bigskip\ 

{\small \underline{2000 Mathematics Subjects Classification: 41A65, 46E30, 46A40 }\hspace{1.5cm}\
\ \ \quad\ \quad . }\smallskip\ 

{\small \underline{Key Words and Phrases:}\hspace{0.15in} Modular spaces, Orlicz spaces, Ces\`{a}ro-Orlicz spaces, uniform monotonicity, strict monotonicity, order continuity.}

\bigskip\ \ 

\section{Introduction}

In 1983, J. Musielak \cite{Mus} published a significant book devoted to semimodular spaces and Orlicz spaces. The book contains basic facts about semimodular and Orlicz spaces equipped with the Amemiya-Orlicz norm and Luxemburg norm. Recently, many authors have investigated intensively geometric properties of the Orlicz spaces  $L^\phi$, Musielak-Orlicz spaces $L^\phi$ and Lorentz spaces $\Lambda_{\phi,w}$ for a Orlicz function $\phi$  and a weight  function $w$ (see e.g. \cite{DoHuLoMaSi,Hudz,hk1,KraRut}). It is worth mentioning that many papers were dedicated especially to rotundity and monotonicity properties which have a deep application to the approximation theory (see e.g. \cite{CheHeHudz,CuHuWi,Hu-Ku}). Also, in the last years the above-mentioned properties were researched with respect to so-called $p$-Amemiya norm (see e.g. \cite{CuDuHuWi,CuHuLiWi,CuHuWi,CuHuWiWl}).
The aim of this paper is to introduce a large class of $s$-norms (norms if $s=1$) which are a natural generalization of the above-recalled norms (see Theorem \ref{increasing}). In the spirit of the previous investigation we study
monotonicity properties, order continuity and auxiliary facts in semimodular spaces. It is worth noticing that we consider not only convex case but we focus on $s$-convex case for $s\in(0,1]$ (for more details see \cite{Mus}).  

The papers is organized as follows. \\ 
Section 2 contains the basic definition and necessary notation which we use in our investigation. \\
The crucial result of Section 3 is Theorem \ref{increasing}, in which we establish a new class of $s$-norms in semimodular spaces. We also research basic properties of these above-mentioned $s$-norms.\\
Section 4 is dedicated to a characterization of order continuity and the Fatou property in $s$-Banach function spaces equipped with the $s$-norms introduced in Theorem \ref{increasing}. Let us mention that the most essential problem that requires a completely new technique is a characterization of the Fatou Property of $s$-Banach function spaces (see Theorem \ref{thm:Fatou}).\\ 
Varies monotonicity properties of $s$-norm introduced in Theorem \ref{increasing} are studied in Section 5. It is worth recalling that the crucial issue of this section is Proposition \ref{prop:modul:mono} which allows us to prove uniform monotonicity (see Theorem \ref{thm:UM}) and strict monotonicity (see Theorem \ref{thm:f:norm:SM}) in $s$-Banach function spaces. At the end of this section we present a complete criteria for monotonicity properties for a certain class of semimodular spaces with respect to above-recalled $s$-norms.

In particular, the results presented in this paper generalized earlier results proved in the case of the Luxemburg and Amemiya-Orlicz norms to a large class of semimodular spaces equipped with $s$-norms. It is necessary to notice that monotonicity properties in the case of $s$-convex semimodulars for $s\in(0,1)$ are very rarely studied in the literature. 
      
\section{Preliminaries}

Let $\mathbb{C}$, $\mathbb{R}$, $\mathbb{R}^+$ and $\mathbb{N}$ be the sets of complex, reals, nonnegative reals and positive integers, respectively. Let us denote by $(e_i)_{i=1}^n$ a standard basis in $\mathbb{R}^n$.
% semimodular space
Let $ X$ be a linear space over $\mathbb{K}$, where $ \mathbb{K} = \mathbb{R}$ or   $ \mathbb{K} = \mathbb{C}.$ A function $ f : X \rightarrow [0, + \infty]$ is called \textit{convex} if
$$f(ax+by) \leq a f(x) + b f(y)$$
for any $ x,y \in X$ and $ a,b \geq 0,$ $a +b =1.$ Let $s\in(0,1].$ A function $ \rho: X \rightarrow [0, + \infty]$ is called an \textit{$s$-convex semimodular} if there holds for arbitrary $ x, y \in X:$
\begin{itemize}
	\item[$(a)$] $\rho(dx) = 0$ for any $ d \geq 0$ implies $x=0$ and $\rho(0)=0;$
	\item[$(b)$] $\rho(dx) = \rho(x),$ for any $ d \in \mathbb{K}, $ $ |d|=1;$
	\item[$(c)$] $\rho(ax+by) \leq a^s \rho(x) + b^s \rho(y)$
	for any $ a,b \geq 0,$ $a^s +b^s =1.$
\end{itemize} 
Put
	$$
	X_{\rho}= \{ x \in X: \rho(dx) < \infty \hbox{ for some } d > 0 \}
	$$
	and 
	$$
	E_{\rho}= \{ x \in X: \rho(dx) < \infty \hbox{ for all } d > 0 \}.
	$$
Then $X_{\rho}$ is called a \textit{semimodular space}. If $ s=1$ $\rho$ is called a convex semimodular. A semimodular $\rho$ on $X_\rho$ is said to be \textit{left-continuous} (resp. \textit{right-continuous}) if for any $x\in{X_\rho}$  and $\lambda_0\in(0,\infty)$ we have 
\begin{equation*}
\lim_{\lambda\rightarrow\lambda_0^{-}}\rho(\lambda{x})=\rho(\lambda_0{x}),\quad\textnormal{and}\quad\left(\textnormal{resp.}\quad\lim_{\lambda\rightarrow\lambda_0^{+}}\rho(\lambda{x})=\rho(\lambda_0{x})\right).
\end{equation*}
We say that a semimodular $\rho$ on $X_\rho$ is \textit{continuous} if it is left- and right-continuous.

A functional $\norm{\cdot}{}{}:X\rightarrow[0,\infty)$ is called $s$-norm for $s\in(0,1]$ if the following conditions are satisfied:
\begin{itemize}
	\item[$(a')$] $\norm{x}{}{}=0$ if and only if $x=0$;
	\item[$(b')$] $\norm{\lambda{x}}{}{}=|\lambda|^s\norm{x}{}{}$ for any $x\in{X}$ and $\lambda\in\mathbb{K}$;
	\item[$(c')$] $\norm{x+y}{}{}\leq\norm{x}{}{}+\norm{y}{}{}$ for all $x,y\in{X}$.
\end{itemize}
% Banach function space
We denote by $\mu$ the Lebesgue measure on $I=[0,\alpha)$, where $\alpha =1$ or $\alpha =\infty$, and by $L^{0}$ the set of all (equivalence classes of) extended real valued Lebesgue measurable functions on $I$.
% For simplicity let us use the notation $A^{c}=I\backslash A$ for any measurable set $A$.
Define by $S_X$ (resp. $B_X)$ the unit sphere (resp. the closed unit ball) in a Banach space $(X,\norm{\cdot}{X}{})$. A Banach lattice $(E,\Vert \cdot \Vert _{E})$ equipped with $s$-norm $\norm{\cdot}{E}{}$, where $s\in(0,1]$, is called an \textit{$s$-Banach function space} (or an \textit{$s$-K\"othe space}) if it is a sublattice of $L^{0}
$ and the following conditions are satisfied
\begin{itemize}
	\item[(1)] If $x\in L^0$, $y\in E$ and $|x|\leq|y|$ a.e., then $x\in E$ and $%
	\|x\|_E\leq\|y\|_E$.
	\item[(2)] There exists a strictly positive $x\in E$.
\end{itemize}
In case when $s=1$, then the space $E$ is called a Banach function space. We use the short notation $E^{+}={\{x \in E:x \ge 0\}}$. We say that $\rho$ a semimodular on a sublattice $X_\rho$ of $L^0$ is \textit{superadditive} if for any $x,y\in{X^+}$ we have 
\begin{equation*}
\rho(x+y)\geq\rho(x)+\rho(y).
\end{equation*}
% OC, Fatou property, UR
An element $x\in E$ is said to be a \textit{point of order continuity} (shortly $x\in{E_a}$) if for any
sequence $(x_{n})\subset{}E^+$ such that $x_{n}\leq \left\vert x\right\vert 
$ and $x_{n}\rightarrow 0$ a.e. we have $\left\Vert x_{n}\right\Vert
_{E}\rightarrow 0.$ An $s$-Banach function space $E$ is called \textit{order continuous} (shortly $E\in \left( OC\right) $) if any element $x\in{}E$ is a point of order continuity. An $s$-Banach function space $E$ is said to have the \textit{Fatou property} if for any $\left( x_{n}\right)\subset{}E^+$, $\sup_{n\in \mathbb{N}}\Vert x_{n}\Vert
_{E}<\infty$ and $x_{n}\uparrow x\in L^{0}$, then $x\in E$ and $\Vert x_{n}\Vert _{E}\uparrow\Vert x\Vert_{E}$. 

% strict monotonicity
A point $x\in{E^+}$ is called a \textit{point of upper monotonicity} (resp. \textit{point of lower monotonicity}) for short a $UM$ \textit{point} (resp. an $LM$ \textit{point}) of $E$ whenever for each $y\in{E^+}$, $x\neq{y}$ with $x\leq{y}$ (resp. $x\neq{y}$ with $y\leq{x}$), we have $\norm{x}{E}{}<\norm{y}{E}{}$ (resp. $\norm{y}{E}{}<\norm{x}{E}{}$). A space $E$ is called \textit{strictly monotone} (shortly $E\in(SM)$) if any element of $E^+$ is a $UM$ point or equivalently if any element of $E^+$ is an $LM$ point.

%  $ULUM$ point
An element $x\in E^{+}$ is said to be a \textit{point of upper local uniform monotonicity} (resp. a \textit{point of lower local uniform monotonicity}), shortly a $ULUM$ \textit{point} (resp. an $LLUM$ \textit{point}), if for any $(x_{n})\subset E^+$ such that $x\leq x_{n}$ and $\left\Vert{}x_{n}\right\Vert _{E}\rightarrow \left\Vert x\right\Vert _{E}$ (resp. $x_n\leq{x}$ and $\norm{x_n}{E}{}\rightarrow\norm{x}{E}{}$), we get $\left\Vert x_{n}-x\right\Vert _{E}\rightarrow 0$. Let us recall that if each
point of $E^{+}\setminus \left\{ 0\right\} $ is a $ULUM$ point (resp. an $LLUM$ point), then we say that $E$ is \textit{upper locally uniformly monotone}, shortly $E\in \left(ULUM\right)$, (resp. \textit{lower locally uniformly monotone}, shortly $E\in \left(LLUM\right)$).

% uniform monotonicity
{An $s$-Banach function space} $E$ is called \textit{uniformly monotone}, shortly $E\in(UM)$, if for any $\epsilon\in(0,1]$ there exists $\delta_\epsilon\in(0,1]$ such that for any $x,y\in{E}^+$ with $x\leq{y}$, $\norm{y}{E}{}=1$ and $\norm{x}{E}{}\geq\epsilon$ we have $\norm{y-x}{E}{}<1-\delta_\epsilon$. 
A space $E$ is called \textit{decreasing uniformly monotone}, shortly $E\in(DUM)$, (resp. \textit{increasing uniformly monotone}, shortly $E\in(IUM)$) if for any $(x_n),(y_n)\subset{E^+}$ such that $x_{n+1}\leq{}x_n\leq{y_n}$ for any $n\in\mathbb{N}$ and $\lim_{n\rightarrow\infty}\norm{x_n}{E}{}=\lim_{n\rightarrow\infty}\norm{y_n}{E}{}<\infty$ (resp. $x_n\leq{y_n}\leq{y_{n+1}}$ for any $n\in\mathbb{N}$ and $\lim_{n\rightarrow\infty}\norm{x_n}{E}{}=\lim_{n\rightarrow\infty}\norm{y_n}{E}{}<\infty$) we have $\norm{x_n-y_n}{E}{}\rightarrow{0}$. For more information see \cite{Hu-Ku}.

% distribution and maximal function
The \textit{distribution function} for any function $x\in L^{0}$ is defined by 
\begin{equation*}
d_{x}(\lambda) =\mu\left\{ s\in [ 0,\alpha) :\left\vert x\left(s\right) \right\vert >\lambda \right\},\qquad\lambda \geq 0.
\end{equation*}
For any function $x\in L^{0}$ its \textit{decreasing rearrangement} is given by 
\begin{equation*}
x^{(r)}\left( t\right) =\inf \left\{ \lambda >0:d_{x}\left( \lambda
\right) \leq t\right\}, \text{ \ \ } t\geq 0.
\end{equation*}
In this article we use the notation $x^{(r)}(\infty)=\lim_{t\rightarrow\infty}x^{(r)}(t)$ if $\alpha=\infty$ and $x^{(r)}(\infty)=0$ if $\alpha=1$. For any function $x\in L^{0}$ we denote the \textit{maximal function} of $x^{(r)}$ by 
\begin{equation*}
x^{\ast \ast }(t)=\frac{1}{t}\int_{0}^{t}x^{(r)}(s)ds.
\end{equation*}
It is well known that for any point $x\in L^{0}$,  $x^{(r)}\leq x^{\ast \ast },$ $x^{\ast \ast }$ is decreasing, continuous and subadditive. For more information of $d_{x}$, $x^{(r)}$ and $x^{\ast \ast }$ see \cite{BS, KPS}. We say that two functions $x,y\in{L^0}$ are \textit{equimeasurable}, shortly $x\sim y$, if $d_x=d_y$. An $s$-Banach function space $(E,\Vert \cdot \Vert_{E}) $ is called \textit{$s$-symmetric} or \textit{$s$-rearrangement invariant} ($s$-r.i. for short) if whenever $x\in L^{0}$ and $y\in E$ such that $x \sim y,$ then $x\in E$ and $\Vert x\Vert_{E}=\Vert y\Vert_{E}$. Let us mention that if $s=1$ we say that $E$ is \textit{symmetric} or \textit{rearrangement invariant} (r.i. for short).

% Orlicz spaces
The function $\psi:\mathbb{R}\rightarrow[0,\infty]$ is called the \textit{Orlicz function} if $\psi$ is nonzero function that is even, convex, continuous and vanishes at zero, $\lim_{\abs{t}{}{}\rightarrow\infty}\psi(t)=\infty$. The mapping $\psi:\mathbb{R}\rightarrow[0,\infty]$ is said to be an \textit{$N$-function} (resp. \textit{$N$-function at zero}) if $\psi$ is even, convex, continuous (resp. even, convex, continuous) and
\begin{equation*}
\lim_{t\rightarrow 0}\frac{\psi(t)}{t}=0\quad\textnormal{and}\quad\lim_{t\rightarrow \infty}\frac{\psi(t)}{t}=\infty\qquad\left(\textnormal{resp.}\quad\lim_{t\rightarrow 0}\frac{\psi(t)}{t}=0\right).
\end{equation*}  
We employ the following parameters
\begin{equation*}
a_\psi=\sup\{t>0:\psi(t)=0\}\quad\textnormal{and}\quad{b_\psi=\sup\{t>0:\psi(t)<\infty\}}.
\end{equation*}
We say that an Orlicz function $\psi$ satisfies condition $\Delta_2$ for all $u\in\mathbb{R^+}$ (shortly $\psi\in\Delta_2$) if there exists $K>0$ such that for all $u\in\mathbb{R}$ we have $\psi(2u)\leq{K}\psi(u)$. It is worth mentioning that if $\psi\in\Delta_2$, then $a_\psi=0$. We define for any Orlicz function $\psi$ its complementary function $\psi_{_Y}$ on $\mathbb{R}$ in the sense of Young and a convex modular $\rho_\psi$ on $L^0$ by
\begin{equation*}
\psi_{_Y}(u)=\sup_{v>0}\{\abs{u}{}{}v-\psi(v)\}\qquad\textnormal{and}\qquad{}I_{\psi}(x)=\int_{I}\psi(x(t))dt
\end{equation*}
for any $u\in\mathbb{R}$ and for any $x\in{L^0}$, respectively. The Orlicz space $L^\psi$ is given by 
\begin{equation*}
L^\psi=\left\{x\in L^0:I_\psi(\lambda x)<\infty,\textnormal{ for some }\lambda>0 \right\}.
\end{equation*}
The Orlicz spaces $L^\psi$ might be considered as Banach spaces equipped with the Luxemburg norm 
\begin{equation*}
\norm{x}{\psi}{}=\inf\left\{\lambda>0:I_\psi\left(\frac{x}{\lambda}\right)\leq{1}\right\}
\end{equation*}
or with the equivalent Orlicz norm
\begin{equation*}
\norm{x}{\psi}{o}=\sup\left\{\abs{\int_{I}y(t)x(t)dt}{}{}:I_{\psi_{_Y}}(y)\leq{1}\right\}.
\end{equation*}
Let us mention that the Orlicz space $L^\psi$ is order continuous if and only if the Orlicz function $\psi$ satisfies condition $\Delta_2$. It is worth mentioning that the Orlicz spaces $L^\psi$ are r.i. Banach function spaces under both the Luxemburg and Orlicz norms (for more details the reader is referred to \cite{BS,KraRut,KPS,Mus}).
% Cesaro function
For any function $x\in L^{0}$ we denote the \textit{Ces\`{a}ro operator} of $x$ by 
\begin{equation}\label{Cesaro:oper}
C(x)(t)=\frac{1}{t}\int_{0}^{t}|x(s)|ds.
\end{equation}
%Orlicz-Lorentz space
Let $w\geq{0}$ be a locally integrable weight function. The weighted Ces\`{a}ro-Orlicz space $C_{\psi,w}$, generated by the Orlicz function $\psi$ and a weight $w$, is a subspace of $L^0$ satisfying
\begin{equation*}
C_{\psi,w}=\left\{x\in{L^0}:I_{\psi,w}(C(\lambda{x}))=\int_{I}\psi\left(C(\lambda{x})(t)\right)w(t)dt<\infty,\textnormal{ for some }\lambda>0\right\}.
\end{equation*}
Ces\`{a}ro function space were researched for the first time in \cite{Shiue}. It is well known that in case when $\psi$ is a power function and $w\geq{0}$ is a weight function, the weighted Ces\`{a}ro-Orlicz space $C_{\psi,w}$ is an order continuous Banach function space with the Fatou property (see \cite{KamKub}).

\section{Construction of various $s$-norms in semimodular space $X_\rho$}

\begin{lemma}
\label{lem1}
Let $ \rho$ be an $s$-convex semimodular. Then for any $ x \in X$ and and $ d \geq 1, $ $ \rho(dx) \geq d^s \rho(x).$ 
\end{lemma}
\begin{proof}
Notice that 
$$
\rho(x) = \rho\left(\frac{dx}{d}\right) = \rho\left(\frac{1}{d}dx+ \left(1-\left(\frac{1}{d}\right)^s\right)^{1/s}0\right) \leq \left(\frac{1}{d}\right)^s \rho(dx),
$$
which shows our claim.
\end{proof}

Now we present the main results of this paper.
\begin{theorem}
\label{increasing}
 Let $ X$ be a linear space over $\mathbb{K},$ $ \mathbb{K} = \mathbb{R}$ or   $ \mathbb{K} = \mathbb{C}$ and let $ s \in (0,1].$ Fix $ n \geq 2$ and let $\rho_i$ be a $s$-convex semimodular defined on $X$ for any $i\in\{1,\dots,n-1\}$. Put $\rho=\max_{1\leq i\leq{n-1}}\{\rho_i\}.$
 Assume that $ f: \mathbb{R}^n \rightarrow [0, +\infty)$ is a convex function such that $f(x)=0$ if and only if $ x=0.$ Assume furthermore that for any  $ x= (1,x_2,\dots,x_{n}) \in (\mathbb{R}_+)^{n}$ and $ y=(1,y_2,\dots,y_{n}) \in (\mathbb{R}_+)^{n}$
 if $ x_j \leq y_j$ for $j=2,\dots,n,$ then
 \begin{equation}
 \label{crucial}
 f(x) \leq f(y).
 \end{equation}
 Let us define for $ x \in X_{\rho},$
 $$
 \|x\|_f = \inf_{k>0}\left\{k f\left(e_1+\sum_{i=2}^n\rho_{i-1}\left(\frac{x}{k^{1/s}}\right)e_i \right) \right\}.
 $$
 Then $ \| \cdot \|_f$ is an $s$-norm (norm if $s=1)$ in $ X_{\rho}.$
\end{theorem} 
\begin{proof}
First observe that $ \| x\|_f \in \mathbb{R}$ for any $ x \in X_{\rho}.$ Fix $ x \in X_{\rho}.$ Then there exists $ d>0$ such that $ \rho(dx) < \infty.$ Then 
$$
kf\left(e_1+\sum_{i=2}^n\rho_{i-1}\left(\frac{x}{k^{1/s}}\right)e_i\right)< \infty,
$$ 
for $ k =1/d,$ which shows that $\| x\|_f \in \mathbb{R}.$
Since $ \lim_{k \rightarrow 0^+} k f(e_1) = 0$ we have $ \| 0 \|_f =0.$ Now fix $ x \in X_{\rho} \setminus\{0\}.$ We show that $ \| x\|_f >0.$ First, we assume that $\rho(\lambda x)\in\{0,\infty\}$ for any $\lambda>0$. Fix $k_0>0$ such that $\rho(x/k_0^{1/s})=\infty$. Next, by \eqref{crucial} we notice that for $k\geq{k_0}$,
\begin{equation*}
kf\left(e_1+\sum_{i=2}^{n}\rho_{i-1}\left(\frac{x}{k^{1/s}}\right)e_i\right) \geq k_0 f(e_1)>0.
\end{equation*}
On the other hand, since $f$ is a convex function and $\rho(x/k^{1/s})=\infty$ for any $k<k_0$ we conclude for any $k<k_0$, 
\begin{equation*}
kf\left(e_1+\sum_{i=2}^n\rho_{i-1}\left(\frac{x}{k^{1/s}}\right)e_i\right)=\infty.
\end{equation*}
Hence, we have $\norm{x}{f}{}\geq{k_0}f(e_1)>0$. Now, assume that $ 0 <\rho(x/(k_0)^{1/s}) < \infty $ for some $ k_0 >0$. Notice that for $ 0<k \leq k_0,$ by $s$-convexity of $\rho_i$ for any $i\in\{1,\dots,n-1\}$, by \eqref{crucial} and by Lemma \ref{lem1} we get
\begin{align*}
kf\left(e_1+\sum_{i=2}^n\rho_{i-1}\left(\frac{x}{k^{1/s}}\right)e_i\right)
&=kf\left(e_1+\sum_{i=2}^n\rho_{i-1}\left(\left(\frac{k_0}{k}\right)^{1/s}\frac{x}{k_0^{1/s}}\right)e_i\right)\\
&\geq
kf\left(e_1+\frac{k_0}{k}\sum_{i=2}^n\rho_{i-1}\left(\frac{x}{k_0^{1/s}}\right)e_i\right)\\
&\geq
k_0f\left(\frac{k}{k_0}e_1+\sum_{i=2}^n\rho_{i-1}\left(\frac{x}{k_0^{1/s}}\right)e_i\right).
\end{align*}
Since $0<\rho(x/k_0^{1/s})<\infty$, we observe that 
\begin{align*}
\lim_{k\rightarrow 0^+}k_0f\left(\frac{k}{k_0}e_1+\sum_{i=2}^n\rho_{i-1}\left(\frac{x}{k_0^{1/s}}\right)e_i\right)
=k_0f\left(\sum_{i=2}^n\rho_{i-1}\left(\frac{x}{k_0^{1/s}}\right)e_i\right)>0.
\end{align*}
Hence, there exists $ \delta > 0$ such that 
\begin{equation*}
kf\left(e_1+\sum_{i=2}^n\rho_{i-1}\left(\frac{x}{k^{1/s}}\right)e_i\right) \geq\frac{k_0}{2}f\left(\sum_{i=2}^n\rho_{i-1}\left(\frac{x}{k_0^{1/s}}\right)e_i\right)
\end{equation*}
for all $ 0 < k \leq \delta.$ Moreover, for any $ k > \delta, $ 
\begin{align*}
kf\left(e_1+\sum_{i=2}^n\rho_{i-1}\left(\frac{x}{k^{1/s}}\right)e_i\right)
&>\delta f\left(e_1+\sum_{i=2}^n\rho_{i-1}\left(\frac{x}{k^{1/s}}\right)e_i\right)\\
&\geq\delta f(e_1)>0.
\end{align*}
Consequently, for any $k>0$ we obtain 
\begin{equation*}
kf\left(e_1+\sum_{i=2}^n\rho_{i-1}\left(\frac{x}{k^{1/s}}\right)e_i\right)
\geq \min\left\{\frac{k_0}{2}f\left(\sum_{i=2}^n\rho_{i-1}\left(\frac{x}{k_0^{1/s}}\right)e_i\right),\delta f(e_1)\right\} >0,
\end{equation*}
which shows that $ \| x\|_f >0.$
\newline
Now we show that $ \| x+y\|_f \leq \|x\|_f + \|y\|_f $ for any $ x,y \in X_{\rho}.$
Fix $\epsilon > 0$ and $x,y \in X_{\rho}.$  By definiton of $ \| \cdot \|_f$ there exists $ u, v >0 $ such that 
$$
u f\left(e_1+\sum_{i=2}^n\rho_{i-1}\left(\frac{x}{u^{1/s}}\right)e_i\right) < \|x\|_f + \epsilon 
$$
and
$$
v f\left(e_1+\sum_{i=2}^n\rho_{i-1}\left(\frac{y}{v^{1/s}}\right)e_i\right)< \|y\|_f + \epsilon . 
$$
Notice that 
\begin{align*}
&(u+v)f\left(e_1+\sum_{i=2}^n\rho_{i-1}\left(\frac{x+y}{(u+v)^{1/s}}\right)e_i\right)\\
&=(u+v)f\left(e_1+\sum_{i=2}^n\rho_{i-1}\left(\frac{ax}{u^{1/s}}+\frac{by}{v^{1/s}}\right)e_i\right),
\end{align*}
where 
$$a=\frac{u^{1/s}}{(u+v)^{1/s}}\qquad\textnormal{and}\qquad{}b=\frac{v^{1/s}}{(u+v)^{1/s}}.$$
Next, since $\rho_j$ is a $s$-convex semimodular for any $j\in\{1,\dots,n-1\}$, by (\ref{crucial}) we get
\begin{align*}
&(u+v)f\left(e_1+\sum_{i=2}^n\rho_{i-1}\left(\frac{ax}{u^{1/s}}+\frac{by}{v^{1/s}}\right)e_i\right)\\ 
&\leq(u+v)a^sf\left(e_1+\sum_{i=2}^n\rho_{i-1}\left(\frac{x}{u^{1/s}}\right)e_i\right)
+(u+v)b^sf\left(e_1+\sum_{i=2}^n\rho_{i-1}\left(\frac{y}{v^{1/s}}\right)e_i\right)\\
&=uf\left(e_1+\sum_{i=2}^n\rho_{i-1}\left(\frac{x}{u^{1/s}}\right)e_i\right)
+vf\left(e_1+\sum_{i=2}^n\rho_{i-1}\left(\frac{y}{v^{1/s}}\right)e_i\right)\\
&<\|x\|_f + \|y\|_f + 2 \epsilon,
\end{align*}
which shows our claim.
\newline
Now, we show that for any $ x \in X_{\rho}$ and $u\in \mathbb{R},$ $ \|ux\|_f = |u|^s \|x\|_f.$ It is obvious that $ \|0x\|_f = |0|^s \|x\|_f =0.$ Hence, we can assume that $u\neq 0.$ Since $\rho_i(z) = \rho_i(-z)$ for any $z\in X_{\rho}$ and $i\in\{1,\dots,n-1\}$, we observe that for any $k>0,$
\begin{align*}
kf\left(e_1+\sum_{i=2}^n\rho_{i-1}\left(\frac{ux}{k^{1/s}}\right)e_i\right)
&=kf\left(e_1+\sum_{i=2}^n\rho_{i-1}\left(\frac{|u|x}{k^{1/s}}\right)e_i\right)\\
&=|u|^s{a}f\left(e_1+\sum_{i=2}^n\rho_{i-1}\left(\frac{x}{a^{1/s}}\right)e_i\right),
\end{align*}
where $ a = {k}/{|u|^s}.$ Then, taking the infimum of the above equality on the left side over all $k>0$ and on the right side over $a>0$ we get our claim. So, the proof is complete.
\end{proof}

Now, applying Theorem\ref{increasing} we get the following theorem.

\begin{theorem}
 \label{increasing1}
Let $ \rho_1$ and $f: \mathbb{R}^2 \rightarrow \mathbb{R}$ be as in Theorem \ref{increasing} and let $ \rho_1$ be an $s$-convex semimodular. Then the function
$$
 \|x\|_f = \inf_{k>0} \{k f(1, \rho_1(x/k^{1/s})) \}
 $$
is an $s$-convex norm (norm if $s=1)$ on $ X_{\rho_1}.$
\end{theorem}
\begin{proof}
It is necessary to apply Theorem \ref{increasing} for $n=2$ 
\end{proof}

\begin{theorem}
\label{equivalent}
Let $f$ and $g$ be two norms on $ \mathbb{R}^n$ satisfying the requirements of Theorem \ref{increasing}. Let $ \rho_1,...,\rho_{n-1}$ and $\rho$ be as in Theorem \ref{increasing}. Then, there are $m,M>0$ such that for any $x\in{X_\rho}$ we have 
\begin{equation*}
m\norm{x}{f}{}\leq\norm{x}{g}{}\leq{M}\norm{x}{f}{}.
\end{equation*}
In particular case when $s=1$, then norms $\norm{\cdot}{f}{}$ and $\norm{\cdot}{g}{}$ are equivalent.
\end{theorem}

\begin{proof}
Since any two norms defined on $ \mathbb{R}^n $ are equivalent, there exists $ m, M > 0$ such that 
$$
mf(\cdot) \leq g(\cdot) \leq M f(\cdot).
$$
Hence for any $ k >0$ and $x \in X_{\rho},$ 
\begin{align*}
mkf\left(e_1+\sum_{i=2}^n\rho_{i-1}\left(\frac{x}{k^{1/s}}\right)e_i\right)
&\leq kg\left(e_1+\sum_{i=2}^n\rho_{i-1}\left(\frac{x}{k^{1/s}}\right)e_i\right)\\
&\leq Mkf\left(e_1+\sum_{i=2}^n\rho_{i-1}\left(\frac{x}{k^{1/s}}\right)e_i\right).
\end{align*}
Taking infimum over $ k > 0,$ we get that
$$
m\| \cdot \|_f \leq \| \cdot \|_g \leq M \| \cdot \|_f, 
$$
as required.
\end{proof}
\begin{corollary}
 \label{monotone}
 Let $ f: \mathbb{R}^n \rightarrow \mathbb{R}$ be a monotone norm on $ \mathbb{R}^n$ and let $ \rho_1,...,\rho_{n-1}$ be $s$-convex semimodulars defined on $X.$ Put $ \rho = \max \{  \rho_1,...,\rho_{n-1}\}.$
 Then the function $ \| \cdot\|_f$ is an  $s$-convex norm (norm if $s=1)$ on $ X_{\rho}.$ In particular, if $n=2$ and $ \rho_1 $ is a convex semimodular then $ \| \cdot \|_f$ is a norm on $ X_{\rho_1}.$ 
\end{corollary}
\begin{proof}
Observe that any monotone norm on $ \mathbb{R}^n$ satisfies (\ref{crucial}). Hence our result follows immediately from Theorem \ref{increasing}.
\end{proof}
\begin{remark}
\label{Luxemburg}
Observe that if $f$ is equal to the maximum norm on $ \mathbb{R}^2,$ $ \rho$ is a convex semimodular and $ x \in X_{\rho},$ then 
$$ 
\| x \|_f = \inf \{u>0: \rho(x/u) \leq 1 \},
$$
which means that $ \| \cdot \|_f$ coincides with the classical Luxemburg norm on $ X_{\rho}.$
If $f$ is equal to the $l_1$-norm on $ \mathbb{R}^2$ and $ \rho$ is a convex semimodular then $ \| \cdot \|_f$ is equal to the classical Orlicz-Amemiya norm on $ X_{\rho},$, which shows that the notion 
of $ \| \cdot \|_f$ is a natural generalization of two classical norms 
considered in semimodular spaces.
Moreover, if $ f$ is the $l_p$-norm on $ \mathbb{R}^2, $ $1 < p < \infty$ and $ \rho$ is a convex semimodular then $ \| \cdot \|_f$ is equal to the $p$-Orlicz-Amemiya norm on $ X_{\rho}$ (see \cite{CuHuWi}).
\end{remark}

\begin{definition}\label{def:regular}
	Assume that $f$ and $X_\rho$ satisfy requirements of Theorem \ref{increasing}. Suppose that for any $x\in{X_\rho}$ and $a>0$ such that $\rho(ax)<\infty$, the function 
	\begin{equation}\label{map:continuous}
	u\rightarrow\rho(ux)
	\end{equation} 
	is continuous in $[0,a]$. We say that $(X_\rho,\norm{\cdot}{f}{})$ is called \textit{regular} if for any $x\in{X_\rho}\setminus\{0\}$ we have $$\liminf_{u\rightarrow{0^+}} uf\left(e_1+\sum_{i=2}^n\rho_{i-1}\left(\frac{x}{u^{1/s}}\right)e_i\right)>\norm{x}{f}{}.$$ 
\end{definition}

\begin{theorem}\label{thm:regular}
	Let $f:\mathbb{R}^n\rightarrow[0,\infty)$ be a convex function satisfying \eqref{crucial} and let $\rho_i$ be a $s$-convex semimodular for every $i\in\{1,\dots,n-1\}$ and $\rho=\max_{1\leq i\leq{n-1}}\{\rho_i\}$. Then the following assertions are satisfied.
	\begin{itemize}
		\item[$(i)$] If $(X_\rho,\norm{\cdot}{f}{})$ is regular, then for any $x\in{X_\rho}$ there exists $k_0>0$ such that 
		\begin{equation}\label{equ:regular}
		\norm{x}{f}{}= k_0f\left(e_1+\sum_{i=2}^n\rho_{i-1}\left(\frac{x}{k_0^{1/s}}\right)e_i\right).
		\end{equation}
		\item[$(ii)$] For any $x\in{X_\rho}\setminus{E_\rho}$ there exists $k_0>0$ such that \eqref{equ:regular} holds.
		\item[$(iii)$] If $f$ is a monotone norm, then for any $x\in{X_\rho}$ either there is $k_0>0$ such that \eqref{equ:regular} is satisfied or 
		\begin{equation*}
		\norm{x}{f}{}=\lim_{k\rightarrow 0^+}{kf\left(\sum_{i=2}^{n}\rho_{i-1}\left(\frac{x}{k^{1/s}}\right)e_i\right)}.
		\end{equation*} 
	\end{itemize}
\end{theorem}

\begin{proof}
	$(i)$. We claim that if $(X_\rho,\norm{\cdot}{f}{})$ is regular, then by the compactness argument for each $x\in{X_\rho}$ there exists $k_0>0$ such that 
	\begin{equation}
	\norm{x}{f}{}= k_0f\left(e_1+\sum_{i=2}^n\rho_{i-1}\left(\frac{x}{k_0^{1/s}}\right)e_i\right).
	\end{equation}
	Indeed, since $f$ satisfies \eqref{crucial} we easily observe that for any $k>0$,
	\begin{equation*}
	kf\left(e_1+\sum_{i=2}^{n}\rho_{i-1}\left(\frac{x}{k^{1/s}}\right)e_i\right)\geq{k}f(e_1)>0.
	\end{equation*}
	Hence, there exists $M_x>0$ such that
	\begin{equation*}
	\norm{x}{f}{}=\inf_{k\in(0,M_x]}\left\{kf\left(e_1+\sum_{i=2}^{n}\rho_{i-1}\left(\frac{x}{k^{1/s}}\right)e_i\right)\right\}.
	\end{equation*}
	Therefore, since $(X_\rho,\norm{\cdot}{f}{})$ is regular there is $0<m_x<M_x$ such that $\rho(x/m_x^{1/s})<\infty$ and
	\begin{equation*}
	\norm{x}{f}{}=\inf_{k\in[m_x,M_x]}\left\{kf\left(e_1+\sum_{i=2}^{n}\rho_{i-1}\left(\frac{x}{k^{1/s}}\right)e_i\right)\right\}.
	\end{equation*}
	Next, by our assumption, the mapping given by
	\begin{equation*}
	k\rightarrow{kf\left(e_1+\sum_{i=2}^{n}\rho_{i-1}\left(\frac{x}{k^{1/s}}\right)e_i\right)}
	\end{equation*}
	is real-valued and continuous in $[m_x,M_x]$. So, there exists $k_0\in[m_x,M_x]$ such that \eqref{equ:regular} is satisfied.\\ 
	$(ii)$. Now, we suppose that $(X_\rho,\norm{\cdot}{f}{})$ is not regular. Then, we notice that for any $x\in{X_\rho}\setminus{E_\rho}$, there exists $k_x>0$ such that $\rho(x/k_x^{1/s})=\infty$, and consequently
	\begin{equation*}
	\lim_{k\rightarrow 0^+}{kf\left(e_1+\sum_{i=2}^{n}\rho_{i-1}\left(\frac{x}{k^{1/s}}\right)e_i\right)}=\infty.
	\end{equation*}
	Hence, there exists $k_0>0$ such that \eqref{equ:regular} holds.\\
	$(ii)$. Now, assume that $f$ is a monotone norm. Then, by the triangle inequality of the norm $f$ we have
	\begin{align}\label{equ:f:norm:triangle}
	-kf(e_1)+kf\left(\sum_{i=2}^{n}\rho_{i-1}\left(\frac{x}{k^{1/s}}\right)e_i\right)
	&\leq
	kf\left(e_1+\sum_{i=2}^{n}\rho_{i-1}\left(\frac{x}{k^{1/s}}\right)e_i\right)\\
	&\leq{}kf\left(\sum_{i=2}^{n}\rho_{i-1}\left(\frac{x}{k^{1/s}}\right)e_i\right)+kf(e_1).\nonumber
	\end{align}
	for any $k>0$ and $x\in{X_\rho}$. Next, by monotonicity and convexity of $f$, by $s$-convexity of $\rho_i$ for all $i\in\{1,\dots,n-1\}$ and by Lemma \ref{lem1} it is easy to see that for any $0<u\leq{}v$,
	\begin{equation*}
	vf\left(\sum_{i=2}^{n}\rho_{i-1}\left(\frac{x}{v^{1/s}}\right)e_i\right)\leq{}uf\left(\sum_{i=2}^{n}\rho_{i-1}\left(\frac{x}{u^{1/s}}\right)e_i\right).
	\end{equation*}
	In consequence, by \eqref{equ:f:norm:triangle} we get
	\begin{equation*}
	\liminf_{k\rightarrow 0^+}{kf\left(e_1+\sum_{i=2}^{n}\rho_{i-1}\left(\frac{x}{k^{1/s}}\right)e_i\right)}=\lim_{k\rightarrow 0^+}{kf\left(\sum_{i=2}^{n}\rho_{i-1}\left(\frac{x}{k^{1/s}}\right)e_i\right)}.
	\end{equation*}
	If $(X_\rho,\norm{\cdot}{f}{})$ is not regular then for any $x\in{X_\rho}$ we have either  \eqref{equ:regular} holds for some $k_0>0$ or 
	\begin{equation*}
	\norm{x}{f}{}=\lim_{k\rightarrow 0^+}{kf\left(\sum_{i=2}^{n}\rho_{i-1}\left(\frac{x}{k^{1/s}}\right)e_i\right)}.
	\end{equation*} 
\end{proof}

\begin{remark}\label{rem:regular:Orlicz}
	Let $f:\mathbb{R}^n\rightarrow[0,\infty)$ be a function as in Theorem \ref{increasing} and let $g_i:\mathbb{R}^+\rightarrow\mathbb{R}^+$ be a convex function such that $g_i(0)=0$, $g_i\neq{0}$ for any $i\in\{1,\dots,n-1\}$. Assume that 
	\begin{equation}\label{N-funtion}
	\lim_{u\rightarrow\infty}\frac{g_1(u)}{u}=\infty.
	\end{equation}
	Fix $s\in(0,1]$. Define for any $i\in\{1,\dots,n-1\}$ and $u\geq{0}$,
    \begin{equation*}
	\phi_i(u)=g_i(u^s),\quad\rho_i(x)=I_{\phi_i}(|x|)\quad\textnormal{ and }\quad\rho=\max_{1\leq i\leq{n-1}}\{\rho_i\}
    \end{equation*}	
	for any $x\in{L^0}$. Observe that $\phi_i$ is an $s$-convex function for $i\in\{1,\dots,n-1\}$. We claim that $(X_\rho,\norm{\cdot}{f}{})$ is regular. Indeed, taking $x\in{X_\rho}\setminus\{0\}$, by convexity of $f$ and \eqref{crucial} we observe that
	\begin{align*}
	\liminf_{u\rightarrow{0^+}}uf\left(e_1+\sum_{k=2}^{n}I_{\phi_{k-1}}\left(\frac{x}{u^{1/s}}\right)e_k\right)
	&\geq\liminf_{u\rightarrow{0^+}}f\left(ue_1+\sum_{k=2}^{n}uI_{\phi_{k-1}}\left(\frac{x}{u^{1/s}}\right)e_k\right)\\
	=\liminf_{t\rightarrow\infty}f\left(\frac{e_1}{t}+\sum_{k=2}^{n}\frac{1}{t}I_{\phi_{k-1}}\left(t^{1/s}x\right)e_k\right)
	&\geq\liminf_{t\rightarrow\infty}f\left(\frac{e_1}{t}+\frac{1}{t}I_{\phi_{1}}\left(t^{1/s}x\right)e_2\right).	
	\end{align*}
	Next, by \eqref{N-funtion} and by Fatou's lemma (see \cite{Royd}) we conclude
	\begin{equation*}
	\liminf_{t\rightarrow\infty}\frac{1}{t}I_{\phi_{1}}\left(t^{1/s}x\right)\geq	\int_{\supp(x)}\liminf_{t\rightarrow\infty}\frac{\phi_{j}(t^{1/s}|x|)}{t}d\mu=\infty.
	\end{equation*}
	Hence, since $x\in{X_\rho}\setminus{0}$, by continuity of $f$ we have
	\begin{align*}
	\liminf_{u\rightarrow{0^+}}uf\left(e_1+\sum_{k=2}^{n}I_{\phi_{k-1}}\left(\frac{x}{u^{1/s}}\right)e_k\right)
	&\geq\liminf_{t\rightarrow\infty}f\left(\frac{e_1}{t}+\frac{1}{t}I_{\phi_{2}}\left(t^{1/s}x\right)e_k\right)\\
	&=\infty>\norm{x}{f}{}. 
	\end{align*}	
	Now, fix $x\in{X_\rho}\setminus\{0\}$ and $a>0$ such that $\rho(ax)<\infty$. Then, by the Lebesgue Dominated Convergence Theorem for $i\in\{1,\dots,n-1\}$ the functions 
	\begin{equation*}
	[0,a]\ni{u}\rightarrow\rho_i(ux)
	\end{equation*}
	are continuous and consequently the mapping 
	\begin{equation*}
	[0,a]\ni{u}\rightarrow\rho(ux)
	\end{equation*}
	is also continuous, which proves our claim.
\end{remark}

\begin{remark}
	Let $s\in(0,1]$ and let $f$ and $\phi_i$ for any $i\in\{1,\dots,n-1\}$ be functions as in Remark \ref{rem:regular:Orlicz}. Assume that $G:L^0\rightarrow\mathbb{R}$ be any operator such that
	\begin{equation*}
	G(ax+by)\leq{a}G(x)+bG(y)\quad\textnormal{ and }\quad{G(ax)=aG(x)}
	\end{equation*}
	for any $a,b\in[0,\infty)$ and $x,y\in{L^0}$. Define for any $i\in\{1,\dots,n-1\}$ and $x\in{L^0}$,
	\begin{equation*}
	\quad\rho_i(x)=I_{\phi_i}(|G(x)|)\quad\textnormal{ and }\quad\rho=\max_{1\leq i\leq{n-1}}\{\rho_i\}
	\end{equation*}	
	for any $x\in{L^0}$. Observe that $\rho_i$ is an $s$-convex semimodular for any $i\in\{1,\dots,n-1\}$. Then, in view of \eqref{N-funtion}, proceeding analogously as in the previous remark we may observe that $(X_\rho,\norm{\cdot}{f}{})$ is regular.
	The typical example of an operator $G$ is the Ces\`{a}ro operator and the maximal function of the decreasing rearrangement. This gives us examples of some natural couples which are regular. 
\end{remark}

The next result give us some information of the space $(X_{\rho, f})^*$ of all linear and continuous with respect to the $\| \cdot \|_f$ functional on $ X_{\rho}.$ Let $\rho$ be a convex semimodular defined on $X$ and let $X'$ denote the space of all linear functionals
defined on $X.$ By (\cite{Mus}, Th.2.3, p. 8), a function $ \rho^* : X' \rightarrow [0, +\infty]$ defined for $ x' \in X'$ by 
$$
\rho^*(x') = \sup \{ |x'(x)| - \rho(x): x \in X_{\rho} \}
$$
is a convex, left-continuous semimodular.

\begin{theorem}
\label{dual1}
 Let $ X$ be a linear space over $\mathbb{K},$ $ \mathbb{K} = \mathbb{R}$ or   $ \mathbb{K} = \mathbb{C}.$ Fix $ n \geq 2$ and let $ \rho_1,...,\rho_{n-1}$ be convex semimodulars defined on $X.$ Put $ \rho = \max \{  \rho_1,...,\rho_{n-1}\}.$
 Assume that $ f: \mathbb{R}^n \rightarrow [0, +\infty)$ is a norm on $ \mathbb{R}^n$ satisfying (\ref{crucial}). Then $((X_{\rho, f})^* = X'_{\rho^*},$ where 
 $$
 X'_{\rho^*}= \{ x' \in X': \rho^*(dx') < \infty \hbox{ for some } d > 0 \}.
$$
\end{theorem} 
\begin{proof}
By Theorem \ref{equivalent}, we can assume that $\| \cdot \|_f$ is equal to the Luxemburg norm on $X_{\rho}.$ By  (\cite{Mus}, Th.2.1, p. 7), $ x' \in (X_{\rho, f})^* $ if and only if there exists $ D >1 $ such that 
$$
|x'(x)| \leq D(\rho(x)+1)
$$
for any $ x \in X_{\rho}.$ Now assume that $ x' \in (X_{\rho, f})^* .$ We show that $ \rho^*(x'/D) < \infty.$
Observe that for any $ x \in X_{\rho},$ by convexity of $ \rho,$ 
$$
|(x'/D)(x)| - \rho(x) \leq D (\rho(x/D)+1) - \rho(x) \leq \rho(x) + D - \rho(x) = D.
$$
Taking supremum over $x \in X_{\rho},$ we get our claim. 
Now assume that $ x' \in X'_{\rho^*}.$ Then  $\rho^*(dx') < \infty$ for some $d > 0.$ Put $D = \rho^*(dx').$ Without loss of generality, we can assume that $D >1.$ Observe that for any $ x \in X_{\rho},$
$$
|dx'(x)|  \leq D + \rho(x) \leq D(\rho(x) +1),
$$ 
which shows that $ dx' \in ((X_{\rho, f})^*   $ and consequently $ x' \in ((X_{\rho, f})^*, $ as required.
\end{proof}
In the next result we estimate the norm in $(X_{\rho, f})^* $ under assumptions that $n=2$ and $f$ is a norm on $\mathbb{R}^n$ satisfying (\ref{crucial}).
\begin{theorem}
\label{dual2}
Let $ X$ be a linear space over $\mathbb{K},$ $ \mathbb{K} = \mathbb{R}$ or   $ \mathbb{K} = \mathbb{C}.$ Let $ \rho$ be a convex semimodular defined on $X.$ 
Assume that $ f: \mathbb{R}^2 \rightarrow [0, +\infty)$ is a norm on $ \mathbb{R}^2$ satisfying (\ref{crucial}) and $f(1,u)=f(u,1)$ for all $u>0$. Then, for any $x^* \in (X_{\rho, f})^*$
$$
\|x^*\|_{f, \rho} = sup \{ |x^*(x)|: x \in X_{\rho}, \|x\|_f =1\} \leq \|x^*\|_{f^*}, 
$$
where 
$$
\|x^*\|_{f^*} = \inf \{k f^*(1, \rho^*(x^*/k)) :k>0 \}
$$
and $ f^*$ is the dual norm to $f$ defined on $ \mathbb{R}^2.$
\end{theorem}
\begin{proof}
Fix $ x \in X_{\rho}$ and   $ x^* \in (X_{\rho,f})^*.$ By Theorem \ref{dual1}, $ x^* \in X'_{\rho^*}.$ Hence we can find $ k>0$ such that $ \rho(x/k) < \infty$ and $ \rho^*(x^*/k) < \infty.$
Observe that 
\begin{align*}
|x^*(x)| = k^2|(x^*/k)(x/k)| &\leq k^2(\rho^*(x^*/k) + \rho(x/k))\\
& \leq k f^*(\rho^*(x/k),1) k f(1,\rho(x/k)).
\end{align*}
Taking infimum over $ k >0,$ we get that 
$$
|x^*(x)| \leq \|x^*\|_{f^*} \|x\|_f,
$$
which shows our claim.
\end{proof}

\begin{corollary}
	Let $X$ be a linear space over $\mathbb{K},$ $ \mathbb{K} = \mathbb{R}$ or   $ \mathbb{K} = \mathbb{C}.$ Let $ \rho$ be a convex semimodular defined on $X.$ Assume that $ f: \mathbb{R}^2 \rightarrow [0, +\infty)$ is a norm on $ \mathbb{R}^2$ satisfying (\ref{crucial}). Then, the norms $\norm{\cdot}{f,\rho}{}$, $\norm{\cdot}{f^*,\rho^*}{}$ and $\norm{\cdot}{\rho^*}{}$ are equivalent, where $f^*$ is the dual norm to $f$ defined on $\mathbb{R}^2$ and
	$$
	\|x^*\|_{f^*, \rho^*} =\inf_{k>0} \{ uf^*\left(1,\rho^*(x^*/u)\right)\}\quad\textnormal{ and }\quad\|x^*\|_{\rho^*} = \inf \{k>0 : \rho^*(x^*/k)\leq{1} \}
	$$
	for any $x^*\in(X_{\rho, f})^*$.
\end{corollary}

\begin{proof}
	Immediately, by Theorem 2.5 in \cite{Mus} it follows that 
	\begin{equation}\label{norm:equival:1}
	\norm{x^*}{\rho^*}{}\leq\norm{x^*}{f,\rho}{}\leq{2}\norm{x^*}{\rho^*}{}
	\end{equation}
	for any $x^*\in{(X_{\rho,f})^*}$. Next, since all norm on $\mathbb{R}^2$ are equivalent, we infer that $f^*$ is equivalent to $A$, where $A(u)=|u_1|+|u_2|$ for any $u\in\mathbb{R}^2$. Hence, by Theorem \ref{equivalent} we conclude that $\norm{\cdot}{f^*,\rho^*}{}$ and $\norm{\cdot}{A,\rho^*}{}$ are equivalent, i.e. there exist $M_1,M_2>0$ such that for any $x^* \in (X_{\rho, f})^*$ we have
	\begin{equation}\label{norm:equival:2}
	M_1\norm{x^*}{f^*,\rho^*}{}\leq\norm{x^*}{A,\rho^*}{}\leq{M_2}\norm{x^*}{f^*,\rho^*}{}.
	\end{equation}
	Moreover, by Theorem 1.10 in \cite{Mus} we obtain that 
	\begin{equation*}
	\norm{x^*}{\rho^*}{}\leq\norm{x^*}{A,\rho^*}{}\leq{2}\norm{x^*}{\rho^*}{}
	\end{equation*}
	 for any $x^* \in (X_{\rho, f})^*$. In consequence, by \eqref{norm:equival:1} and \eqref{norm:equival:2} we complete the proof.
\end{proof}

\section{order continuity and the fatou property of banach function space $X_\rho$}

Immediately, by Theorem \ref{equivalent} we get the following result.

\begin{corollary}\label{equival:f&A}
  Let $X$ be a linear space over $ \mathbb{K},$ $\mathbb{K} = \mathbb{R}$ or   $ \mathbb{K} = \mathbb{C}.$ Let $ n \geq 2$, $s\in(0,1]$ and $ \rho_1,...,\rho_{n-1}$ be $s$-convex semimodulars defined on $X$ and $\rho = \max \{  \rho_1,...,\rho_{n-1}\}.$
 Assume that $ f: \mathbb{R}^n \rightarrow [0, +\infty)$ is a norm on $ \mathbb{R}^n$ satisfying (\ref{crucial}). Let $A:\mathbb{R}^n\rightarrow[0,\infty)$ be a functional given by $A(x)=\abs{x_1}{}{}+\max_{2{\leq}i\leq{n}}\abs{x_i}{}{}$ for any $x\in\mathbb{R}^n$. Then, the norms $\norm{\cdot}{f}{}$ and $\norm{\cdot}{A}{}$ are equivalent.
\end{corollary}

\begin{proposition}\label{proposistion:BFS}
	Let $X$ be a linear subspace of $L^0$ and $s\in(0,1]$, $n\geq{2}$ and $\rho_1,\dots,\rho_{n-1}$ be $s$-convex semimodulars on $X$. Assume that $f:\mathbb{R}^n\rightarrow[0,\infty)$ is a convex function satisfying \eqref{crucial} and $\rho=\max_{1\leq i\leq{n-1}}\{\rho_i\}$. Let $X_\rho=\{x\in{X}:\rho(\lambda{x})<\infty\textnormal{ for some }\lambda\}$ and
	\begin{itemize}
		\item[$(i)$] For any $x\in X_\rho$ and $y\in{L^0}$ such that $\abs{y}{}{}\leq\abs{x}{}{}$ a.e., we have $y\in{X_\rho}$ and $\rho_i(\lambda y)\leq\rho_i(\lambda x)$ for all $\lambda>0$ and $i\in\{1,\dots,n-1\}$.
		\item[$(ii)$] There is $x\in{X_\rho}$ such that $x>0$ a.e.
		\item[$(iii)$] For any $(x_m)\subset{X_\rho}$ such that $\rho(\lambda(x_k-x_j))\rightarrow{0}$ as $k,j\rightarrow\infty$ for any $\lambda>0$, then there is $x\in{X_\rho}$ and $\rho(\lambda(x_m-x))\rightarrow{0}$ as $n\rightarrow\infty$ for any $\lambda>0$. 
	\end{itemize}
	
	If $(i)-(iii)$ are satisfied, then $X_\rho$ is an $s$-Banach function space equipped with the $s$-norm given by
	\begin{equation*}
	\norm{x}{f}{}=\inf_{k>0}\left\{kf\left(e_1+\sum_{i=2}^n\rho_{i-1}\left(\frac{x}{k^{1/s}}\right)e_i\right)\right\}
	\end{equation*} 
	for any $x\in{X_\rho}$. Additionally, $(x_m)\subset{X_\rho}$ is a Cauchy sequence with respect to $\norm{\cdot}{f}{}$ if and only if for any $\lambda>0$ we have $\rho(\lambda(x_k-x_j))\rightarrow{0}$ as $k,j\rightarrow\infty$. 
\end{proposition}

\begin{proof}
	First, observe that $X_\rho$ is a Banach lattice. Indeed, for any $x,y,z\in{X_\rho}$ and $a\in\mathbb{R}^+$ such that $x\leq{y}$ a.e. we have $x+z\leq{y+z}$ a.e. Moreover, if $x\geq{0}$ a.e., then $ax\geq{0}$ a.e. Next, by condition $(i)$, for any $x\in{X_\rho}$, $y\in{L^0}$ such that $|y|\leq|x|$ a.e. we get 
	\begin{equation*}
	\rho(\lambda y)=\max_{1\leq i\leq{n-1}}\{\rho_i(\lambda y)\}\leq\max_{1\leq i\leq{n-1}}\{\rho_i(\lambda x)\}=\rho(\lambda x)
	\end{equation*}
	for every $\lambda>0$, consequently $y\in{X_\rho}$. Now, we show that $x\vee{y}\in{X_\rho}$ for any $x,y\in{X_\rho}$. It is easy to see that $\rho(\lambda x)=\rho(\lambda|x|)$ for any $\lambda>0$, whence $|x|\in{X_\rho}$, and analogously $|y|\in{X_\rho}$. Since $|x\vee{y}|\leq|x|\vee|y|\leq|x|+|y|$ a.e. Therefore, since $|x|+|y|\in{X_\rho}$ we have $|x\vee y|\in{X_\rho}$. Thus, we observe $x\vee{y}\in{X_\rho}$. We claim that the norm $\norm{\cdot}{f}{}$ is monotone, i.e. for any $x,y\in{X_\rho}$ such that $|x|\leq|y|$ a.e. we have $\norm{x}{f}{}\leq\norm{y}{f}{}$. Indeed, by conditions \eqref{crucial} and $(i)$ it follows that 
	\begin{align*}
	\norm{x}{f}{}&=\inf_{k>0}\left\{kf\left(e_1+\sum_{i=2}^n\rho_{i-1}\left(\frac{x}{k^{1/s}}\right)e_i\right)\right\}\\
	&\leq\inf_{k>0}\left\{kf\left(e_1+\sum_{i=2}^n\rho_{i-1}\left(\frac{y}{k^{1/s}}\right)e_i\right)\right\}=\norm{y}{f}{}.
	\end{align*} 
	Finally, by Theorems 1.6 and 1.10 in \cite{Mus} and by Corollary \ref{equival:f&A} we conclude that for any $(y_m)\subset{X_\rho}$ we have $\norm{y_m}{f}{}\rightarrow{0}$ if and only if $\rho(\lambda(y_m))\rightarrow{0}$ as $m\rightarrow\infty$ for all $\lambda>0$ (or equivalently $\rho_i(\lambda(y_m))\rightarrow{0}$ as $m\rightarrow\infty$ for all $\lambda>0$ and $i\in\{1,\dots,n-1\}$).
	Therefore, $(x_m)\subset{X_\rho}$ is a Cauchy sequence in $X_\rho$ with respect to $\norm{\cdot}{f}{}$ if and only if $\rho(\lambda(x_j-x_k))\rightarrow{0}$ as $k,j\rightarrow\infty$ for all $\lambda>0$. Hence, by condition $(iii)$, we infer that $(X_\rho,\norm{\cdot}{f}{})$ is complete and finish the proof.	 
\end{proof}

\begin{lemma}	
	Let $f$, $X_\rho$ and $\rho_i$ for any $i\in\{1,\dots,n-1\}$ satisfy requirements of Proposition \ref{proposistion:BFS}. If for any $(x_m)\subset{X_\rho}$ and $x\in{L^0}$ we have
	\begin{equation*}
	\liminf_{m\rightarrow\infty}|x_m|\geq|x|\quad\textnormal{a.e.}\quad\Rightarrow\quad\liminf_{m\rightarrow\infty}\rho_i(\lambda{}x_m)\geq\rho_i(\lambda{}x)
	\end{equation*}
	for all $i\in\{1,\dots,n-1\}$ and $\lambda>0$, then $\rho_i$ for any $i\in\{1,\dots,n-1\}$ and $\rho$ are left-continuous on $(0,\infty)$ and continuous at zero. 
\end{lemma}

\begin{proof}
	First, we easily observe that $\rho_i$ for any $i\in\{1,\dots,n-1\}$ and $\rho$ are continuous at zero. Indeed, taking $x\in{X_\rho}$, there exists $\beta>0$ such that $\rho(\beta{x})<\infty$, and for any $(\lambda_m)\subset\mathbb{R}\setminus\{0\}$ such that $1\geq|\lambda_m|\downarrow{0}$, by Lemma \ref{lem1} we have 
	\begin{equation*}
	\rho_i(\lambda_m\beta x)\leq\rho(\lambda_m\beta{x})\leq|\lambda_m|^s\rho(\beta x)\rightarrow{0}\quad\textnormal{as}\quad{m\rightarrow\infty}
	\end{equation*}
	for any $i\in\{1,\dots,n-1\}$. Now, we show that $\rho_i$ for every $i\in\{1,\dots,n-1\}$ is left-continuous on $(0,\infty)$. Suppose for a contrary that there exist $\lambda_0>0$, $x\in{X_\rho}$ and $i\in\{1,\dots,n-1\}$ such that 
	\begin{equation*}
	\lim_{\lambda\rightarrow\lambda_0^-}\rho(\lambda{x})<\rho_i(\lambda_0{x})<\infty.
	\end{equation*}
	Since $\rho_i(\lambda{x})\leq\rho_i(\beta{x})$ for any $\lambda\leq\beta$, for any sequence $(\lambda_m)\subset\mathbb{R}^+$ such that $\lambda_m\uparrow\lambda_0$ we get
	\begin{equation}\label{equ:continu:rho}
	\rho(\lambda_m{x})\uparrow\lim_{\lambda\rightarrow\lambda_0^-}\rho_i(\lambda{x}).
	\end{equation}
	Next, by assumption that for any sequence $(y_m)\subset{X_\rho}$ and $y\in{L^0}$,
	\begin{equation*}
	\liminf_{m\rightarrow\infty}|y_m|\geq|y|\quad\textnormal{a.e.}\quad\Rightarrow\quad\liminf_{m\rightarrow\infty}\rho_i(\lambda{}y_m)\geq\rho_i(\lambda{}y),
	\end{equation*}
	replacing $y_m$ and $y$ by $\lambda_m{x}$ and $\lambda_0{x}$ respectively, we obtain
	\begin{equation*}
	\liminf_{m\rightarrow\infty}\rho_i(\lambda_m{x})\geq\rho_i(\lambda_0{x})>\lim_{\lambda\rightarrow\lambda_0^-}\rho_i(\lambda{x})\geq\rho_i(\lambda_m{x})
	\end{equation*}  
	for any $m\in\mathbb{N}$. In consequence, by \eqref{equ:continu:rho} we have a contradiction, which proves left-continuity of $\rho_i$ for any $i\in\{1,\dots,n-1\}$ and $\rho$ on $(0,\infty)$. 
\end{proof}

\begin{theorem}\label{thm:Fatou}
	Let $f:\mathbb{R}^n\rightarrow[0,\infty)$ be a monotone norm and let $X_\rho$ and $\rho_1,\cdots,\rho_{n-1}$ be continuous and satisfy requirements of Proposition \ref{proposistion:BFS}. If for any $(x_m)\subset({X_\rho})$ and $x\in{L^0}$ we have
	\begin{equation*}
	\liminf_{m\rightarrow\infty}|x_m|\geq|x|\quad\textnormal{a.e.}\quad\Rightarrow\quad\liminf_{m\rightarrow\infty}\rho_i(\lambda{}x_m)\geq\rho_i(\lambda{}x)
	\end{equation*}
	for all $i\in\{1,\dots,n-1\}$ and $\lambda>0$, then $X_\rho$ has the Fatou property.
\end{theorem}

\begin{proof}
	Let $(x_m)\subset{X_\rho^+}$ and $x\in{L^0}$ be such that ${x_m}\leq{x}$ for all $m\in\mathbb{N}$ and $x_m\uparrow{x}$ a.e. and $\sup_{m\in\mathbb{N}}\norm{x_m}{f}{}<\infty$. Then, by Corollary \ref{equival:f&A} there is $A_1>0$ such that for any $m\in\mathbb{N}$
	\begin{equation*}
	\norm{x_m}{A}{}\leq{A_1}\norm{x_m}{f}{}\leq{A_1}\sup_{m\in\mathbb{N}}\norm{x_m}{f}{}.
	\end{equation*}
	Hence, taking $M>A_1\sup_{m\in\mathbb{N}}\norm{x_m}{f}{}$, by Theorem 1.10 in \cite{Mus} we have
	\begin{equation*}
	\norm{x_m}{\rho}{s}\leq\norm{x_m}{A}{}<M
	\end{equation*}
	for every $m\in\mathbb{N}$. Thus, we obtain
	\begin{equation}\label{equ1:Fatou}
	\rho_i\left(\frac{x_m}{M^{1/s}}\right)\leq\rho\left(\frac{x_m}{M^{1/s}}\right)\leq{1}
	\end{equation}
	for any $m\in\mathbb{N}$ and $i\in\{1,\dots,n-1\}$. Moreover, since $|x_m|\leq|x|$ a.e. for every $m\in\mathbb{N}$, by monotonicity of $\rho_i$ for all $i\in\{1,\dots,n-1\}$ we get 
	\begin{equation}\label{equ2:Fatou}
	\rho_i(\lambda x_m)\leq\rho_i(\lambda x)	
	\end{equation}
	for any $\lambda>0$, $m\in\mathbb{N}$ and $i\in\{1,\dots,n-1\}$. Next, since $x_m\uparrow{x}$ a.e., it is clear that $\liminf_{m\rightarrow\infty}|x_m|\geq|x|$ a.e. and consequently, 
	\begin{equation*}
	\liminf_{m\rightarrow\infty}\rho_i(\lambda x_m)\geq\rho_i(\lambda x)
	\end{equation*}
	for every $\lambda>0$ and $i\in\{1,\dots,n-1\}$. Hence, by \eqref{equ2:Fatou} we conclude
	\begin{equation}\label{equ3:Fatou}
	\rho_i(\lambda x_m)\uparrow\rho_i(\lambda x)\quad\textnormal{as}\quad{m\rightarrow\infty}
	\end{equation}
	for any $\lambda>0$ and $i\in\{1,\dots,n-1\}$. Therefore, by \eqref{equ1:Fatou} we obtain
	\begin{equation*}
	\rho\left(\frac{x}{M^{1/s}}\right)=\max_{1\leq i\leq{n-1}}\left\{\rho_i\left(\frac{x}{M^{1/s}}\right)\right\}\leq{1},
	\end{equation*}
	whence $x\in{X_\rho}$. Next, since $f$ is a monotone and convex on $\mathbb{R}^n$, by \eqref{equ3:Fatou} it follows that
	\begin{equation}\label{equ:FP:converg}
	kf\left(e_1+\sum_{i=2}^n\rho_{i-1}\left(\frac{x_m}{k^{1/s}}\right)e_i\right)\uparrow kf\left(e_1+\sum_{i=2}^n\rho_{i-1}\left(\frac{x}{k^{1/s}}\right)e_i\right)\quad\textnormal{as}\quad{}m\rightarrow\infty
	\end{equation}
	for any $k>0$. Moreover, by definition of the norm $\norm{\cdot}{f}{}$ it is easy to see that 
	\begin{equation}\label{equ:FP:inequa}
	\norm{x}{f}{}\leq{}kf\left(e_1+\sum_{i=2}^n\rho_{i-1}\left(\frac{x}{k^{1/s}}\right)e_i\right)
	\end{equation}
	for all $k>0$. Since $\norm{x_m}{f}{}\leq\norm{x_{m+1}}{f}{}\leq\norm{x}{f}{}$ for all $m\in\mathbb{N}$ and $d=\sup_{m\in\mathbb{N}}\norm{x_m}{f}{} <\infty$ we conclude $\norm{x_m}{f}{}\uparrow{d}$, whence for any $\epsilon\in(0,d/2)$ there exists $N_\epsilon\in\mathbb{N}$ such that for any $m\geq{N_\epsilon}$ we get
	\begin{equation}\label{equ:1:FP}
	d-\epsilon\leq\norm{x_m}{f}{}\leq\norm{x}{f}{}.
	\end{equation}
	Let $m\geq{N_\epsilon}$. Then, assuming that $(e_i)_{i=1}^n$ is standard basis in $\mathbb{R}^n$, by definition of a norm $\norm{\cdot}{f}{}$ there exists $k_1>0$ such that
	\begin{align*}
	\norm{x_m}{f}{}&\leq k_1f\left(e_1+\sum_{i=2}^{n}\rho_{i-1}\left(\frac{x_m}{k_1^{1/s}}\right)e_i\right)
	\leq{}k_1f\left(e_1+\sum_{i=2}^{n}\rho_{i-1}\left(\frac{x_{m+1}}{k_1^{1/s}}\right)e_i\right)\\
	&\leq\norm{x_{m+1}}{f}{}+\frac{1}{2}\leq{d}+\frac{1}{2}.
	\end{align*}
	Similarly, we may find $k_2>0$ such that
	\begin{align*}
	\norm{x_m}{f}{}&\leq k_2f\left(e_1+\sum_{i=2}^{n}\rho_{i-1}\left(\frac{x_m}{k_2^{1/s}}\right)e_i\right)
	\leq{}k_2f\left(e_1+\sum_{i=2}^{n}\rho_{i-1}\left(\frac{x_{m+1}}{k_2^{1/s}}\right)e_i\right)\\
	&\leq{}k_2f\left(e_1+\sum_{i=2}^{n}\rho_{i-1}\left(\frac{x_{m+2}}{k_2^{1/s}}\right)e_i\right)\\
	&\leq\norm{x_{m+2}}{f}{}+\frac{1}{4}\leq{d}+\frac{1}{4}.
	\end{align*}
	Next, by mathematical induction, there is a sequence $(k_j)\subset\mathbb{R}^+\setminus\{0\}$ such that 
	\begin{align}\label{equ:2:FP}
	\norm{x_m}{f}{}&\leq k_jf\left(e_1+\sum_{i=2}^{n}\rho_{i-1}\left(\frac{x_m}{k_j^{1/s}}\right)e_i\right)\leq{}k_jf\left(e_1+\sum_{i=2}^{n}\rho_{i-1}\left(\frac{x_{m+1}}{k_j^{1/s}}\right)e_i\right)\\
	&\leq\dots\leq{}k_jf\left(e_1+\sum_{i=2}^{n}\rho_{i-1}\left(\frac{x_{m+j}}{k_j^{1/s}}\right)e_i\right)\nonumber\\
	&\leq\norm{x_{m+j}}{f}{}+\frac{1}{2^j}\leq{d}+\frac{1}{2^j}.\nonumber
	\end{align}
	Now, we claim that $(k_j)$ is bounded. Indeed, if it is not true, then passing to subsequence and relabeling if necessary we may assume $k_j\rightarrow\infty$.	Next, by \eqref{equ:2:FP} we obtain
	\begin{equation*}
	d+\frac{1}{2^j}\geq{}k_jf\left(e_1+\sum_{i=2}^{n}\rho_{i-1}\left(\frac{x_m}{k_j^{1/s}}\right)e_i\right)\geq{k_j}f(e_1)>0.
	\end{equation*}
	and consequently, since $d<\infty$ we have a contradiction. So, $(k_j)$ has a convergent subsequence. Then, passing to subsequence and relabeling if necessary we may assume that $k_j\rightarrow{k_0}\in\mathbb{R}^+$. Now, we continue the proof in two cases.\\
	\textit{Case $1$.} Suppose that $k_0>0$. Therefore, by continuity of $f$ and $\rho_i$ for any $i\in\{1,\dots,n-1\}$ and by \eqref{equ:2:FP} we get
	\begin{equation*}
	\norm{x_m}{f}{}\leq
	k_0f\left(e_1+\sum_{i=2}^{n}\rho_{i-1}\left(\frac{x_m}{k_0^{1/s}}\right)e_i\right)\leq{}d
	\end{equation*}
	for any $m\geq{N_\epsilon}$. Hence, by \eqref{equ:1:FP} it follows that
	\begin{equation*}
	k_0f\left(e_1+\sum_{i=2}^{n}\rho_{i-1}\left(\frac{x_m}{k_0^{1/s}}\right)e_i\right)\rightarrow d\quad\textnormal{as}\quad{m\rightarrow\infty}.
	\end{equation*}
	Furthermore, by \eqref{equ:FP:converg} and \eqref{equ:FP:inequa}, without loss of generality we may assume that for any $m\geq{N_\epsilon}$, 
	\begin{equation*}
	\norm{x}{f}{}\leq{}k_0f\left(e_1+\sum_{i=2}^{n}\rho_{i-1}\left(\frac{x_m}{k_0^{1/s}}\right)e_i\right)+\epsilon.
	\end{equation*}
	In consequence, we have $\norm{x}{f}{}\leq{d}$. On the other hand, since $x_m\leq{x}$ a.e. for every $m\in\mathbb{N}$, by monotonicity of the norm $\norm{\cdot}{f}{}$ this yields that $\norm{x_m}{f}{}\leq\norm{x}{f}{}$ for all $m\in\mathbb{N}$. Thus, by definition of $d$ we conclude $\norm{x}{f}{}=d$.\\
	\textit{Case $2$.} Now, we assume that ${k_0}=0$. Without loss of generality, passing to subsequence and relabeling if necessary we may suppose that $k_j\downarrow{0}$. Then, since $\rho_i$ is $s$-convex for any $i\in\{1,\dots,n-1\}$, by Lemma \ref{lem1} we notice that for any $0<u\leq{v}$,
	\begin{equation*}
	v\rho_i\left(\frac{x}{v^{1/s}}\right)\leq{u}\rho_i\left(\frac{x}{u^{1/s}}\right)
	\end{equation*}
	for any $i\in\{1,\dots,n-1\}$. Next, since $f$ is convex and monotone, it is easy to see that for any $0<u\leq{v}$,
	\begin{equation}\label{equ:5:FP}
	vf\left(\sum_{i=2}^{n}\rho_{i-1}\left(\frac{x}{v^{1/s}}\right)e_i\right)\leq{}uf\left(\sum_{i=2}^{n}\rho_{i-1}\left(\frac{x}{u^{1/s}}\right)e_i\right).
	\end{equation}
	In view of the above inequality, since $k_j\downarrow{0}$, we may define  
	\begin{equation*}
	c=\lim_{j\rightarrow\infty}k_jf\left(\sum_{i=2}^{n}\rho_{i-1}\left(\frac{x}{k_j^{1/s}}\right)e_i\right)\quad\textnormal{and}\quad 	c_m=\lim_{j\rightarrow\infty}k_jf\left(\sum_{i=2}^{n}\rho_{i-1}\left(\frac{x_m}{k_j^{1/s}}\right)e_i\right).
	\end{equation*}
	for any $m\geq{N_\epsilon}$. Furthermore, since $f$ is monotone, by \eqref{equ:2:FP} and \eqref{equ:5:FP} we get
	\begin{equation*}
	d\geq{c_m}\geq k_jf\left(\sum_{i=2}^{n}\rho_{i-1}\left(\frac{x_m}{k_j^{1/s}}\right)e_i\right)
	\end{equation*}
	for all $m\geq{N_\epsilon}$ and $j\in\mathbb{N}$. Hence, since $k_j>0$ for every $j\in\mathbb{N}$, by \eqref{equ:FP:converg} we obtain
	\begin{equation*}
	d\geq k_jf\left(\sum_{i=2}^{n}\rho_{i-1}\left(\frac{x}{k_j^{1/s}}\right)e_i\right)
	\end{equation*}
	for any $j\in\mathbb{N}$. Consequently, by definition of $c$ it follows that $c\leq{d}$. Furthermore, by the triangle inequality of the norm $f$, we easily observe that for any $j\in\mathbb{N}$,
	\begin{align*}
	-k_jf(e_1)+k_jf\left(\sum_{i=2}^{n}\rho_{i-1}\left(\frac{x}{k_j^{1/s}}\right)e_i\right)
	&\leq
	k_jf\left(e_1+\sum_{i=2}^{n}\rho_{i-1}\left(\frac{x}{k_j^{1/s}}\right)e_i\right)\\
	&\leq{}k_jf\left(\sum_{i=2}^{n}\rho_{i-1}\left(\frac{x}{k_j^{1/s}}\right)e_i\right)+k_jf(e_1).
	\end{align*}
	Therefore, by \eqref{equ:FP:inequa} and by definition of $c$ we have
	\begin{equation*}
	d\geq{c}=\lim_{j\rightarrow\infty}k_jf\left(e_1+\sum_{i=2}^{n}\rho_{i-1}\left(\frac{x}{k_j^{1/s}}\right)e_i\right)\geq\norm{x}{f}{}.
	\end{equation*}
	Finally, since $\norm{x_{m_k}}{f}{}\leq\norm{x}{f}{}$ for any $k\in\mathbb{N}$ and by assumption that $\norm{x_{m_k}}{f}{}\uparrow{d}$ we infer $\norm{x_{m_k}}{f}{}\uparrow\norm{x}{f}{}$ and finish case $2$.	
\end{proof}

Now, we discuss a complete characterization of order continuity in the space $X_\rho$.

\begin{theorem}\label{thm:OC-space}
	Let $X_\rho$, $f$ and $\rho_1,...,\rho_{n-1}$ satisfy requirements of Proposition \ref{proposistion:BFS}. Additionally, assume that $ f: \mathbb{R}^n \rightarrow [0, +\infty)$ is a norm on $ \mathbb{R}^n$ satisfying (\ref{crucial}). The $s$-norm $\norm{\cdot}{f}{}$ is order continuous if and only if for any $\lambda>0$ and $(x_m)\subset({X_\rho})^+$, $x_m\downarrow{0}$ a.e. we have $\rho(\lambda{x_m})\rightarrow{0}$.
\end{theorem}

\begin{proof}
	Immediately, by Corollary \ref{equival:f&A} we have the norms $\norm{\cdot}{f}{}$ and $\norm{\cdot}{A}{}$ are equivalent. Moreover, by Theorem 1.10 in \cite{Mus} the norms $\norm{\cdot}{\rho}{s}$ and $\norm{\cdot}{A}{}$ are equivalent, where  $\norm{x}{\rho}{s}=\inf\{u>0:\rho(x/u^{1/s})\leq{1}\}$ for any $x\in{X_\rho}$.	Next, taking $(x_m)\subset{X_\rho}$, by Theorem 1.6 in \cite{Mus} it is well known that $\norm{x_m}{\rho}{s}\rightarrow{0}$ if and only if $\rho(x_m\lambda)\rightarrow{0}$ for every $\lambda>0$. Therefore, assuming that $(x_m)\subset{(X_\rho)^+}$ and $x_m\downarrow{0}$ a.e. we easily observe that $\norm{x_m}{f}{}\rightarrow{0}$ if and only if $\rho(x_m\lambda)\rightarrow{0}$ for any $\lambda>0$.
\end{proof}

Now, we show an example of semimodular space $X_\rho$ that is order continuous.

\begin{example}\label{example:Cesaro:OC}
	Let $f:\mathbb{R}^n\rightarrow[0,\infty)$ be a convex function satisfying \eqref{crucial}, $\phi_{i}$ be an Orlicz function satisfying $\Delta_2$ condition and let $w_i>{0}$ be a weight function such that $w_i\in{D_{\phi_i}}$ for any $i\in\{1,\dots,n-1\}$, i.e. for any $0\leq{a<b}<\infty$ and $i\in\{1,\dots,n-1\}$ we have 
	\begin{equation}\label{necessar:cond:2}
	\int_{a}^{b}\phi\left(\frac{t-a}{t}\right)w_i(t)dt<\infty,\quad\textnormal{ and }\quad\int_{b}^{\infty}\phi_i\left(\frac{b-a}{t}\right)w_i(t)dt<\infty.
	\end{equation}
	Define for any $x\in{L^0}$ and $i\in\{1,\dots,n-1\}$,
	\begin{equation*}
	\rho_i(x)=\int_I\phi_i(C(x)(t))w_i(t)dt\qquad\textnormal{and}\qquad\rho(x)=\max_{1\leq i\leq{n-1}}\{\rho_i(x)\},
	\end{equation*}
	where $C$ is the Ces\`{a}ro operator given by \eqref{Cesaro:oper}. Since $\phi_i$ is convex and $w_i>{0}$ for all $i\in\{1,\dots,n-1\}$, by subadditivity of the Ces\`{a}ro operator we easily observe that $\rho_i$ is convex, continuous, superadditive and monotone semimodular for any $i\in\{1,\dots,n-1\}$. Furthermore, by \eqref{necessar:cond:2} we have  
	\begin{align*}
	X_\rho&=\bigcap_{i=1}^{n-1}\left\{x\in{L^0}:\int_I\phi_i(\lambda{C(x)(t)})w_i(t)dt<\infty,\textnormal{ for some }\lambda>0\right\}\\
	&=\bigcap_{i=1}^{n-1}C_{\phi_i,w_i}\neq\{0\}.
	\end{align*}	
	Furthermore, by Proposition \ref{proposistion:BFS} we conclude that ${X_\rho}$ is a Banach function space.
	Now, we claim that $X_\rho$ is order continuous. Indeed, picking $i\in\{1,\dots,n-1\}$ and taking $(x_m)\subset{(C_{\phi_i,w_i})^+}$ such that $x_m\downarrow{0}$ a.e., since $\phi_i$ satisfies $\Delta_2$ condition, we can easily observe that $\rho_{i}(\lambda{x_m})<\infty$ for all $m\in\mathbb{N}$ and $\lambda>0$. Moreover, since $\rho_i(\lambda x_1)<\infty$ for any $\lambda>0$, applying twice Lebesgue Dominated Convergence Theorem (see \cite{Royd}) we obtain $\rho_i(\lambda{x_m})\rightarrow{0}$ for any $\lambda>0$, and consequently by definition of $\rho$ we infer $\rho(\lambda{x_m})\rightarrow{0}$ for all $\lambda>0$. Finally, by Theorem \ref{thm:OC-space} we prove our claim.
\end{example}

\begin{lemma}
	Let $X_\rho$, $f$ and $\rho_1,...,\rho_{n-1}$ satisfy requirements of Proposition \ref{proposistion:BFS}. The following condition are equivalent.
	\begin{itemize}
		\item[$(i)$] For any $\lambda>0$ and $(x_m)\subset({X_\rho})^+$ with $x_m\downarrow{0}$ a.e. we have $\rho(\lambda{x_m})\rightarrow{0}$.
		\item[$(ii)$] There is $\eta>1$ such that for any $(x_m)\subset({X_\rho})^+$ with $x_m\downarrow{0}$ a.e. we have $\rho(\eta{x_m})\rightarrow{0}$.
	\end{itemize}
\end{lemma}

\begin{proof}
	Clearly, it is easy to see that $(i)\Rightarrow(ii)$. Now, we prove that $(ii)\Rightarrow(i)$. Let $\lambda>0$ and let $(x_m)\subset({X_\rho})^+$ with $x_m\downarrow{0}$ a.e. Suppose that there exists $\eta>1$ which satisfies $(ii)$. First, observe that if $\lambda\leq\eta$, then $\rho(\lambda{y_m})\leq\rho(\eta{x_m})\rightarrow{0}$. In case when $\lambda>\eta$, we may find $k\in\mathbb{N}\setminus\{1\}$ such that $\eta^k>\lambda$. Define $y_m=\eta^{k-1}x_m$ for any $m\in\mathbb{N}$. Notice that $(y_m)\subset({X_\rho})^+$ and $y_m\downarrow{0}$ a.e. Hence, by $(ii)$ we obtain $\rho(\eta{y_m})\rightarrow{0}$. Finally, we have
	\begin{equation*}
	\rho(\lambda{x_m})\leq\rho(\eta\eta^{k-1}{x_m})=\rho(\eta{y_m})\rightarrow{0}.
	\end{equation*}
\end{proof}

\begin{remark}\label{rem:rho-converg}
	Now, let us recall well known fact that for any sequence $(x_m)\subset{X_\rho}$ the following condition holds 
	\begin{equation}\label{rho-converg}
	\rho(x_m)\rightarrow{0}\quad\Rightarrow\quad\rho(2x_m)\rightarrow{0}
	\end{equation}
	if and only if the norm convergence $\norm{\cdot}{\rho}{s}$ in $X_\rho$ is equivalent to $\rho$-convergence in $X_\rho$, i.e. for any $(x_n)\subset{X_\rho}$ we have $\norm{x_n}{\rho}{s}\rightarrow{0}$ if and only if there is $\lambda>0$ such that $\rho(\lambda x_n)\rightarrow{0}$ (see Property 5.2 in \cite{Mus}).  Next, assuming that $f$ is a norm satisfying \eqref{crucial} on $\mathbb{R}^n$, by Theorem 1.10 in \cite{Mus} and by Corollary \ref{equival:f&A} we may easily replace the norm $\norm{\cdot}{\rho}{s}$ by the norm $\norm{\cdot}{f}{}$.
\end{remark}

\begin{proposition}\label{prop:FP=>OC}
	Let $X_\rho$, $f$ and $\rho_1,...,\rho_{n-1}$ satisfy requirements of Proposition \ref{proposistion:BFS}. If $\rho_i$ is superadditive and satisfies \eqref{rho-converg} for any $i\in\{1,\dots,n-1\}$ and also for any $(z_m)\subset{X_\rho}$ and $z\in{X_\rho}$ we have
	\begin{equation*}
	\liminf_{m\rightarrow\infty}|z_m|\geq{|z|}\textnormal{ a.e. }\quad\Rightarrow\quad\liminf_{m\rightarrow\infty}\rho_i(z_m)\geq{\rho_i(z)},
	\end{equation*}
	for every $i\in\{1,\dots,n-1\}$, then $X_\rho$ is order continuous.
\end{proposition}

\begin{proof}
	Let $(x_m)\subset({X_\rho})^+$ be such that $x_m\downarrow{0}$ a.e. Define $y=x_1$, $y_m=x_1-x_m$ for all $m\in\mathbb{N}\setminus\{1\}$. Then, by superadditivity of $\rho_i$ we conclude
	\begin{equation}\label{equ:FP=>OC}
	\rho_i(y-y_m)\leq\rho_i(y)-\rho_i(y_m)
	\end{equation}
	for every $m\in\mathbb{N}$ and $i\in\{1,\dots,n-1\}$. Moreover, since $y_m\uparrow{y}$ a.e., by monotonicity of $\rho_i$ we get $\rho_i(y_m)\leq\rho_i(y)$ for all $m\in\mathbb{N}$, $i\in\{1,\dots,n-1\}$ and also $\liminf_{m\rightarrow\infty}y_m\geq{y}$ a.e. Hence, we infer that 
	\begin{equation*}
	\liminf_{m\rightarrow\infty}\rho_i(y_m)\geq{\rho_i(y)},
	\end{equation*}
	and consequently 
	\begin{equation*}
	\lim_{m\rightarrow\infty}\rho_i(y_m)={\rho_i(y)}
	\end{equation*}
	for all $i\in\{1,\dots,n-1\}$. In consequence, since $\rho_i$ satisfies \eqref{rho-converg} for any $i\in\{1,\dots,n-1\}$, by \eqref{equ:FP=>OC} for any $\lambda>0$ it follows that 
	\begin{equation*}
	\rho(\lambda x_m)=\rho(\lambda(y-y_m))=\max_{1\leq i\leq{n-1}}\left\{\rho_i(\lambda(y-y_m))\right\}\rightarrow{0}.
	\end{equation*}
	Finally, by Theorem \ref{thm:OC-space} we conclude that $X_\rho$ is order continuous.
\end{proof}

\begin{theorem}
	Let $f$ be a monotone norm on $\mathbb{R}^n$ and let $X_\rho$ and $\rho_1,\cdots,\rho_{n-1}$ satisfy requirements of Proposition \ref{proposistion:BFS}. The following assertions are satisfied. If $x\in{X_\rho}$ and $\norm{x}{f}{}<1$, then $$\rho(c^{1/s}x)\leq\norm{x}{f}{},$$ where $c=\min_{1\leq i\leq{n}}f(e_i)$ and $(e_i)_{i=1}^n$ is standard basis in $\mathbb{R}^n$.
\end{theorem}

\begin{proof}
	Let $x\in{X_\rho}$, $\epsilon>0$ and $\norm{x}{f}{}+\epsilon<1$. Then, by definition of $\norm{\cdot}{f}{}$ there exists $k_0>0$ such that
	\begin{equation*}
	k_0f\left(e_1+\sum_{i=2}^n\rho_{i-1}\left(\frac{x}{k_0^{1/s}}\right)e_{i}\right)\leq\norm{x}{f}{}+\epsilon<1.
	\end{equation*}
	Fix $i\in\{1,\dots,n-1\}$. Then, we have
	\begin{equation*}
	k_0f\left(e_1+\sum_{i=2}^n\rho_{i-1}\left(\frac{x}{k_0^{1/s}}\right)e_{i}\right)\geq 	k_0f\left(e_1+\rho_{i}\left(\frac{x}{k_0^{1/s}}\right)e_i\right).
	\end{equation*}
	Next, since $c=\min_{1\leq i\leq{n}}f(e_i)>0$ we obtain
	\begin{equation*}
 	k_0f\left(e_1+\rho_{i}\left(\frac{x}{k_0^{1/s}}\right)e_i\right)\geq{c}k_0\max\left\{\rho_{i}\left(\frac{x}{k_0^{1/s}}\right),1\right\}.
	\end{equation*}	
	Hence, we have
	\begin{equation*}
	ck_0\rho_{i}\left(\frac{x}{k_0^{1/s}}\right)\leq ck_0\max\left\{\rho_{i}\left(\frac{x}{k_0^{1/s}}\right),1\right\}\leq\norm{x}{f}{}+\epsilon<1.
	\end{equation*}	
	Therefore, since $ck_0<1$, by $s$-convexity of $\rho_i$ we get  
	\begin{equation*}
	\rho_{i}\left(c^{1/s}x\right)\leq ck_0\rho_{i}\left(\frac{x}{k_0^{1/s}}\right)\leq\norm{x}{f}{}+\epsilon.
	\end{equation*}	
	Next, since $i\in\{1,\dots,n-1\}$  and $\epsilon>0$ are arbitrary, we obtain $\rho(c^{1/s}x)\leq\norm{x}{f}{}$.
\end{proof}

\begin{proposition}\label{prop:OC:Orlicz}
	Let $f:\mathbb{R}^n\rightarrow[0,\infty)$ be a convex function satisfying \eqref{crucial} and let $\phi_i$ be an Orlicz function and $\rho_i(x)=I_{\phi_i}(|x|)$ for any $i\in\{1,\dots,n-1\}$ and $x\in{X_\rho}$, where $\rho=\max_{1\leq i\leq{n-1}}\{\rho_i\}$. The Orlicz function $\max_{1\leq i\leq{n-1}}\{\phi_i\}$ satisfies $\Delta_2$ condition if and only if the norm $\norm{\cdot}{f}{}$ is order continuous (equivalently $X_\rho=E_\rho=\bigcap_{i=1}^{n-1}E^{\phi_i}$). 
\end{proposition}

\begin{proof}	
	First, by Theorem \ref{thm:OC-space} we observe that the norm $\norm{\cdot}{f}{}$ is order continuous if and only if for any $(x_m)\subset({X_\rho})^+$, $x\in({X_\rho})^+$ such that $x_m\leq{x}$ a.e. for any $m\in\mathbb{N}$ and $x_m\rightarrow{0}$ a.e. we have $\rho(\lambda{x_m})\rightarrow{0}$ for all $\lambda>0$. So, equivalently we may write that for any $(x_m)\subset({X_\rho})^+$, $x\in({X_\rho})^+$ such that $x_m\leq{x}$ a.e. for any $m\in\mathbb{N}$ and $x_m\rightarrow{0}$ a.e., we get $I_{\phi_i}(\lambda{x_m})\rightarrow{0}$ for any $i\in\{1,\dots,n-1\}$ and $\lambda>0$. Let $\phi=\max_{1\leq i\leq{n-1}}\{\phi_{i}\}$. Then, in view of Theorem 10.3 \cite{KraRut} it follows that the Orlicz space $L^{\phi}=E^{\phi}$ (i.e. $L^{\phi}$ is order continuous) if and only if $\phi$ satisfies $\Delta_2$ condition. Next, by Theorem 1 \cite{Hudz} it is easy to notice that
	\begin{equation*}
	X_\rho=\left\{x\in{L^0}:\max_{1\leq{}i\leq{n-1}}\{I_{\phi_{i}}(\lambda{x})\}<\infty\textnormal{ for some }\lambda>0\right\}=\bigcap_{i=1}^{n-1}L^{\phi_i}=L^\phi=E^\phi.
	\end{equation*}
	Moreover, for any $\lambda>0$ and $x\in{E^\phi}$ it is easy to observe that 
	\begin{align*}
	\max_{1\leq i\leq{n-1}}\{I_{\phi_i}(\lambda x)\}\leq\int_I\max_{1\leq i\leq{n-1}}\{\phi_i(\lambda x)\}d\mu={I_\phi(\lambda x)}<\infty,
	\end{align*}
	whence 
	\begin{equation*}
	E^\phi\subset\bigcap_{i=1}^{n-1}E^{\phi_i}=\left\{x\in{L^0}:\max_{1\leq{}i\leq{n-1}}\{I_{\phi_{i}}(\lambda{x})\}<\infty\textnormal{ for all }\lambda>0\right\}=E_\rho.
	\end{equation*}
	In consequence, since $E^{\phi_i}\subset{L^{\phi_i}}$ for any $i\in\{1,\dots,n-1\}$, we have
	\begin{equation*}
	X_\rho=E_\rho=\bigcap_{i=1}^{n-1}E^{\phi_i}.
	\end{equation*}
	Finally, for every $(x_m)\subset({X_\rho})^+$, $x\in({X_\rho})^+$ such that $x_m\leq{x}$ a.e. for any $m\in\mathbb{N}$ and $x_m\rightarrow{0}$ a.e., by Lebesgue Dominated Convergence Theorem we have $I_{\phi_i}(\lambda{x_m})\rightarrow{0}$ for any $\lambda>0$ and consequently $\norm{x_m}{f}{}\rightarrow{0}$.
\end{proof}

\section{Monotonicity properties of banach function space $X_\rho$}

This section is devoted to monotonicity properties in $s$-Banach function space. We start our investigation with a characterization of strict monotonicity. In this section we assume unless we say otherwise that $X\rho$ is an $s$-Banach function space satisfying requirements of Proposition \ref{proposistion:BFS}.

\begin{theorem}\label{thm:SM}
	Let $f:\mathbb{R}^n\rightarrow[0,\infty)$ be a convex function that is strictly monotone on $(\mathbb{R}^n)^+$ and let $(X_\rho,\norm{\cdot}{f}{})$ be regular (see Definition \ref{def:regular}). If there is $j\in\{1,\dots,n-1\}$ such that for any $x,y\in({X_\rho})^+$ with $x\leq{y}$, $x\neq y$ we have $\rho_j(\lambda x)<\rho_j(\lambda y)$ whenever $\rho_j(\lambda y)<\infty$ for $\lambda>0$, then the space $X_\rho$ equipped with the norm $\norm{\cdot}{f}{}$ is strictly monotone.
\end{theorem}

\begin{proof}
	 Let $x,y\in({X_\rho})^+$ be such that $x\leq{y}$, $x\neq{y}$ a.e. Clearly, we have $\rho_i(\lambda x)\leq\rho_i(\lambda y)$ for all $\lambda>0$ and $i\in\{1,\dots,n-1\}$. Moreover, by assumption there is $j\in\{1,\dots,n-1\}$ such that if $\rho_j(\lambda y)<\infty$ for $\lambda>0$, then $\rho_j(\lambda x)<\rho_j(\lambda y)$. Hence, by strict monotonicity of $f$ we get
	 \begin{equation}\label{equ:SM}
	 kf\left(e_1+\sum_{i=2}^n\rho_{i-1}\left(\frac{x}{k^{1/s}}\right)e_i\right)<kf\left(e_1+\sum_{i=2}^n\rho_{i-1}\left(\frac{y}{k^{1/s}}\right)e_i\right)
	 \end{equation}
	 for any $k>0$ such that $\rho(y/k^{1/s})<\infty$. Moreover, since $(X_\rho,\norm{\cdot}{f}{})$ is regular, by Theorem \ref{thm:regular}, there exists $k_0>0$ such that 
	 \begin{equation*}
	 \norm{y}{f}{}= k_0f\left(e_1+\sum_{i=2}^n\rho_{i-1}\left(\frac{y}{k_0^{1/s}}\right)e_i\right).
	 \end{equation*}
	 In consequence, by definition of the norm $\norm{\cdot}{f}{}$ and by \eqref{equ:SM} we obtain
	 \begin{align*}
	 \norm{x}{f}{}&\leq{}k_0f\left(e_1+\sum_{i=2}^n\rho_{i-1}\left(\frac{x}{k_0^{1/s}}\right)e_i\right)\\
	 &<k_0f\left(e_1+\sum_{i=2}^n\rho_{i-1}\left(\frac{y}{k_0^{1/s}}\right)e_i\right)=\norm{y}{f}{},
	 \end{align*}
	 which proves that $X_\rho$ is strictly monotone.
\end{proof}

We present the all details of the proof of the following proposition for the sake of completeness and reader's convenience. Let us mention that in some part of the below proof we use similar techniques to Lemma 2.3 in \cite{CuHuWi}.   

\begin{proposition}\label{prop:modul:mono}
	Let $f$ be a monotone norm on $\mathbb{R}^n$. Let $\rho_i$ be a convex, left-continuous, superadditive and monotone semimodular on $X$ for any $i\in\{1,\dots,n-1\}$. Then, for any $x,y\in{X_\rho}$ such that $\norm{y}{f}{}=1$ and $0\leq{x}\leq{y}$ a.e. we have $\rho(cx)\leq{1}$ and
	\begin{equation*}
	\norm{y-x}{f}{}\leq{1}-\delta_f(\rho(cx)),
	\end{equation*}
	where $c=\min_{1\leq i\leq{n}}\{f(e_i)\}$, $(e_i)_{i=1}^n$ elements of standard basis in $\mathbb{R}^n$, and $\delta_f$ is a modulus of monotonicity of a function $f$ given by
	\begin{equation*}
	\delta_f(\epsilon)=\inf\left\{1-f(u-v):u,v\in(\mathbb{R}^n)^+,u\leq{v},f(v)=1,f(u)\geq\epsilon\right\},
	\end{equation*}
	for any $\epsilon\in[0,1]$.
\end{proposition}

\begin{proof}
	First, we easily observe that $c=\min_{1\leq i\leq{n}}\{f(e_i)\}>0$. Then, by definition of a norm $\norm{\cdot}{f}{}$ there is $k_0>0$ such that 
	\begin{align}\label{equ:delta}
	ck_0\leq k_0f(e_1)&\leq k_0f\left(e_1+\sum_{i=2}^n\rho_{i-1}\left(\frac{y}{k_0}\right)e_i\right)\\
	&\leq\norm{y}{f}{}+\epsilon=1+\epsilon.\nonumber
	\end{align}
	Fix $j\in\{1,\dots,n-1\}$. For simplicity of our notation we define 
	\begin{equation*}
	v=k_0e_1+\sum_{k=2}^{n}k_0\rho_{k-1}\left(\frac{y}{k_0}\right)e_k,\quad\textnormal{and}\quad u=k_0\rho_j\left(\frac{x}{k_0}\right)e_j.
	\end{equation*}
	Then, since a mapping $k\rho_i(x/k)$ is nonincreasing with respect $k>0$ for any $i\{1,\dots,n-1\}$, by monotonicity of $f$ we get
	\begin{align}\label{equ2:delta}
	f(v)&\geq f\left(\sum_{i=2}^nk_0\rho_{i-1}\left(\frac{x}{k_0}\right)e_i\right)\geq f(u)=k_0\rho_j\left(\frac{x}{k_0}\right)f(e_j)\\
	&\geq ck_0\rho_j\left(\frac{cx}{ck_0}\right)\geq(1+\epsilon)\rho_j\left(\frac{cx}{1+\epsilon}\right).\nonumber
	\end{align}
	Moreover, since $\rho_i$ is superadditive for any $i\in\{1,\dots,n-1\}$ we obtain
	\begin{align*}
	\norm{y-x}{f}{}&\leq k_0f\left(e_1+\sum_{k=2}^{n}\rho_{k-1}\left(\frac{y-x}{k_0}\right)e_k\right)\\
	&\leq k_0f\left(e_1+\sum_{k=2}^{n}\rho_{k-1}\left(\frac{y}{k_0}\right)e_k-\rho_{k-1}\left(\frac{x}{k_0}\right)e_k\right)\\
	&= f\left(v-\sum_{k=2}^{n}k_0\rho_{k-1}\left(\frac{x}{k_0}\right)e_k\right).
	\end{align*}	
	Consequently, since $u\leq \sum_{k=2}^{n}k_0\rho_{k-1}\left(\frac{x}{k_0}\right)e_k\leq{v}$, by monotonicity of $f$ and by monotonicity of $\rho_i$ for any $i\in\{1,\dots,n-1\}$ we conclude 
	\begin{equation*}
	\norm{y-x}{f}{}\leq f\left(v-\sum_{k=2}^{n}k_0\rho_{k-1}\left(\frac{x}{k_0}\right)e_k\right)\leq f\left(v-u\right).
	\end{equation*}
	Next, since $u\leq{v}$, by \eqref{equ:delta} and \eqref{equ2:delta} it is easy to see that 
	\begin{equation*}
	1\geq f\left(\frac{v}{1+\epsilon}\right)\geq f\left(\frac{u}{1+\epsilon}\right)\geq{\rho_j\left(\frac{cx}{1+\epsilon}\right)}.
	\end{equation*}
	Therefore, by definition of the modulus of monotonicity $\delta_f$ we have
	\begin{equation*}
	\norm{y-x}{f}{}\leq f\left(v-u\right)\leq(1+\epsilon)\left\{1-\delta_f\left(\rho_j\left(\frac{cx}{1+\epsilon}\right)\right)\right\}.
	\end{equation*}
	Finally, since $\epsilon>0$ and $j\in\{1,\dots,n-1\}$ are arbitrary, by left-continuity of $\rho_i$ for any $i\in\{1,\dots,n-1\}$, we observe that
	\begin{equation*}
	\rho(cx)\leq{1}\quad\textnormal{and}\quad\norm{y-x}{f}{}\leq 1-\delta_f\left(\rho\left({cx}\right)\right).
	\end{equation*}
\end{proof}

\begin{theorem}\label{thm:UM}
	Let $f$ be a norm on $\mathbb{R}^n$ that is uniformly monotone. Let $\rho_i$ be a convex, left-continuous, superadditive and monotone semimodular on $X$ for any $i\in\{1,\dots,n-1\}$. If $\rho$ satisfies \eqref{rho-converg}, then the space $X_\rho$ equipped with the norm $\norm{\cdot}{f}{}$ is uniformly monotone. 
\end{theorem}

\begin{proof}
	Let $\epsilon\in(0,1]$ and let $x,y\in(X_\rho)^+$ be such that $x\leq{y}$ a.e. and $\norm{y}{f}{}=1$, $\norm{x}{f}{}\geq\epsilon$. We claim that for any $\epsilon>0$ and for any $\lambda>0$ we have 
	$$\psi_f(\lambda,\epsilon)=\inf\{\rho(\lambda z):z\in{X_\rho}, \norm{z}{f}{}\geq\epsilon\}>0.$$
	Indeed, if we suppose for a contrary that it is not true, then there exist $\epsilon_1>0$ and $\lambda_1>0$ such that $\psi_f(\lambda_1,\epsilon_1)=0$. Hence, there exists $(x_m)\subset{X_\rho}$ such that $\rho(\lambda_1 x_m)<\frac{1}{m}$ and $\norm{x_m}{f}{}\geq\epsilon_1$ for all $m\in\mathbb{N}$. Consequently, since $\rho$ satisfies \eqref{rho-converg}, it follows that $\rho(\lambda{x_m})\rightarrow{0}$ for any $\lambda>0$. In consequence, by Remark \ref{rem:rho-converg} we get a contradiction, which proves the claim. Next, since $\rho_i$ is left-continuous and superadditive for any $i\in\{1,\dots,n-1\}$, by Proposition \ref{prop:modul:mono} it follows that $1\geq\rho(cx)\geq{\psi_f(c,\epsilon)}>0$ and
	\begin{equation*}
	\norm{y-x}{f}{}\leq 1-\delta_f(\rho(cx))\leq 1-\delta_f(\psi_f(c,\epsilon)).
	\end{equation*}
	Therefore, by assumption that $f$ is uniformly monotone we have $$1\geq\delta_f(\psi_f(c,\epsilon))>0$$ 
	for any $\epsilon\in(0,1]$ and $c>0$, which shows that $\norm{\cdot}{f}{}$ is uniformly monotone.
\end{proof}

Now, we present an example of semimodular spaces that are uniformly monotone.  

\begin{example}
	Let $f:\mathbb{R}^n\rightarrow[0,\infty)$ be uniformly monotone norm, $\phi_{i}$ be an Orlicz function satisfying $\Delta_2$ condition and let $w_i\in{D_{\phi_i}}$ be a positive weight function for any $i\in\{1,\dots,n-1\}$ (for more details please see \eqref{necessar:cond:2}).
	Defining $\rho$ and $\rho_i$ for any  $i\in\{1,\dots,n-1\}$ as in Example \ref{example:Cesaro:OC} we infer that $\rho_i$ is convex, continuous, superadditive and monotone semimodular for any $i\in\{1,\dots,n-1\}$ and   
	\begin{align*}
	X_\rho&=\bigcap_{i=1}^{n-1}C_{\phi_i,w_i}\neq\{0\}.
	\end{align*}	
	Next, we observe that $\rho$ satisfies \eqref{rho-converg}. So, by Theorem \ref{thm:UM} we obtain $X_\rho$ equipped with $\norm{\cdot}{f}{}$ is uniformly monotone.
\end{example}

\begin{proposition}\label{prop:NOT:SM}
	Let $f:\mathbb{R}^n\rightarrow[0,\infty)$ be a convex function satisfying \eqref{crucial} and let $\phi_i$ be an Orlicz function and $\rho_i(x)=I_{\phi_i}(|x|)$ for any $i\in\{1,\dots,n-1\}$ and $x\in{X_\rho}$, where $\rho=\max_{1\leq i\leq{n-1}}\{\rho_i\}$. If each $\phi_i$ vanishes outside zero, then $X_\rho$ equipped with the norm $\norm{\cdot}{f}{}$ is not strictly monotone.
\end{proposition}

\begin{proof}
	Assume that $a_{\phi_i}>0$ for any $i\in\{1,\dots,n-1\}$. Let $\epsilon>0$ and $y\in{S(X_\rho)}^+$ be such that $\mu(I\setminus\supp(y))>0$. Then, there exists $k_0>0$ such that 
	\begin{equation*}
	k_0f\left(e_1+\sum_{k=2}^{n}\rho_{k-1}\left(\frac{y}{k_0}\right)e_k\right)\leq\norm{y}{f}{}+\epsilon.
	\end{equation*}
	Let $A\subset{I\setminus\supp(y)}$ and $0<\mu(A)<\infty$. Define
	\begin{equation*}
	x=y+\min_{1\leq i\leq{n-1}}\{a_{\phi_i}\}k_0\chi_A.
	\end{equation*}
	Notice that $0\leq{y}\leq{x}$ and $y\neq{x}$, whence $\norm{y}{f}{}\leq\norm{x}{f}{}$. On the other hand, by assumption that $\rho_i$ vanishes outside zero for any $i\in\{1,\dots,n-1\}$ we have
	\begin{align*}
	\norm{x}{f}{}&\leq 	k_0f\left(e_1+\sum_{k=2}^{n}\rho_{k-1}\left(\frac{x}{k_0}\right)e_k\right)\\
	&= k_0f\left(e_1+\sum_{k=2}^{n}I_{\phi_{k-1}}\left(\frac{y}{k_0}+{\min_{1\leq i\leq{n-1}}\{a_{\phi_i}\}\chi_A}\right)e_k\right)\\
	&= k_0f\left(e_1+\sum_{k=2}^{n}I_{\phi_{k-1}}\left(\frac{y}{k_0}\right)e_k\right)\leq\norm{y}{f}{}+\epsilon.
	\end{align*}
	Hence, we get $\norm{x}{f}{}=\norm{y}{f}{}$, consequently by assumption that $x\neq y$ and $y\leq{x}$ a.e. we conclude that $X_\rho$ is not strictly monotone.
\end{proof}

Now, we discuss a complete characterization of strict monotonicity for space $X_\rho$. First, we present a full criteria of strict monotonicity for $X_\rho$ in case when $f$ is strictly monotone and convex function and $(X_\rho,\norm{\cdot}{f}{})$ is regular (see Definition \ref{def:regular}). Next, we provide an equivalent condition for $SM$ in case when $f$ is a uniformly monotone norm.
 
\begin{theorem}\label{thm:f:convex:SM}
	Let $f:\mathbb{R}^n\rightarrow[0,\infty)$ be a convex function that is strictly monotone on $(\mathbb{R}^n)^+$. Let $\phi_i$ be an Orlicz function such that  and $\rho_i(x)=I_{\phi_i}(|x|)$ for any $i\in\{1,\dots,n-1\}$ and $x\in{X_\rho}$, where $\rho=\max_{1\leq i\leq{n-1}}\{\rho_i\}$. Assume that there is $k\in\{1,\dots,n-1\}$ such that $\phi_k$ is $N$-function at $\infty$. The space $X_\rho$ equipped with the norm $\norm{\cdot}{f}{}$ is strictly monotone if and only if $a_{\phi_j}=0$ for some $j\in\{1,\dots,n-1\}$.
\end{theorem}

\begin{proof}
	Immediately, since there exists $k\in\{1,\dots,n-1\}$ such that $\phi_k$ is $N$-function at $\infty$, by Remark \ref{rem:regular:Orlicz} we infer that $(X_\rho,\norm{\cdot}{f}{})$ is regular. Next, in view of assumption that $f$ is a strictly monotone norm, since a semimodular $\rho_i$ is monotone for any $i\in\{1,\dots,n-1\}$ and by Proposition \ref{prop:NOT:SM} and Theorem \ref{thm:SM} we complete the proof.
\end{proof}

\begin{theorem}\label{thm:f:norm:SM}
Let $f$ be a norm on $\mathbb{R}^n$ that is uniformly monotone. Let $\phi_i$ be an Orlicz function such that  and $\rho_i(x)=I_{\phi_i}(|x|)$ for any $i\in\{1,\dots,n-1\}$ and $x\in{X_\rho}$, where $\rho=\max_{1\leq i\leq{n-1}}\{\rho_i\}$. The space $X_\rho$ equipped with the norm $\norm{\cdot}{f}{}$ is strictly monotone if and only if $a_{\phi_j}=0$ for some $j\in\{1,\dots,n-1\}$.
\end{theorem}

\begin{proof}
	\textit{Necessity}. Immediately, by Proposition \ref{prop:NOT:SM} we get that there exists $j\in\{1,\dots,n-1\}$ such that $a_{\phi_j}=0$.\\
	\textit{Sufficiency}. Suppose that $a_{\phi_j}=0$ for some $j\in\{1,\dots,n-1\}$. Let $x,y\in{X_\rho}$ be such that $0\leq{x}\leq{y}$ and $x\neq{y}$ and let $c=\min_{1\leq i\leq{n}}\{f(e_i)\}$, where $(e_i)_{i=1}^n$ are elements of standard basis in $\mathbb{R}^n$. We may consider without loss of generality that $\norm{y}{f}{}=1$. Clearly, we have $0\leq{}y-x\leq{y}$ and $y\neq y-x\neq{0}$. 
	Then, since $I_{\phi_i}$ is left-continuous and superadditive for any $i\in\{1,\dots,n-1\}$, by Proposition \ref{prop:modul:mono} and by assumption that $\phi_j$ vanishes only at zero for some $j\in\{1,\dots,n-1\}$ we obtain $1\geq\rho(cx)\geq{I_{\phi_j}(cx)}>0$ and consequently by assumption that $f$ is uniformly monotone we get 
	\begin{equation*}
	\norm{y-x}{f}{}\leq{1}-\delta_f(\rho(cx))<1=\norm{y}{f}{}.
	\end{equation*}
\end{proof}

We say that a sequence $(x_m)$ in a Banach space $X$ is nearly order convergence to $x\in{X^+}$ if $x_m\leq{x}$ for any $n\in\mathbb{N}$ and $\norm{x_m}{X}{}\uparrow\norm{x}{X}{}$. Now, we show analogous result to Lemma 3.2 in \cite{CuHuWi}. For a sake of completeness and the reader's convenience we present the whole details of the following proof. 

\begin{lemma}\label{lem:conv:in:measu}
	Let $f$ be a norm on $\mathbb{R}^n$ that is uniformly monotone. Let $\phi_i$ be an Orlicz function and $\rho_i(x)=I_{\phi_i}(|x|)$ for any $i\in\{1,\dots,n-1\}$ and $x\in{X_\rho}$, where $\rho=\max_{1\leq i\leq{n-1}}\{\rho_i\}$. If there is $j\in\{1,\dots,n-1\}$ such that $a_{\phi_j}=0$, then the nearly order convergence implies the convergence in measure on $(X_\rho)^+$.
\end{lemma}

\begin{proof}
	Let $(x_m)\subset{(X_\rho)^+}$, $x\in({X_\rho})^+$ be such that  $x_m\leq{x}$ for any $m\in\mathbb{N}$ and $\norm{x_m}{x}{}\uparrow\norm{x}{X}{}$. Without loss of generality we may assume that $\norm{x}{f}{}=1$. Suppose for a contrary that $x_m$ does not converge to $x$ in measure, i.e. there exist $\epsilon,\eta>0$ and $(m_k)\subset\mathbb{N}$ such that $\inf_{k\in\mathbb{N}}\mu(A_k)\geq\eta$ where
	\begin{equation*}
	A_k\subset\left\{t\in{I}:x(t)-x_{m_k}(t)>\epsilon\right\}\quad\textnormal{and}\quad{}\mu(A_k)<\infty,
	\end{equation*}
	for any $k\in\mathbb{N}$. Let $c=\min_{1\leq i\leq{n}}\{f(e_i)\}$, where $(e_i)_{i=1}^n$ is a standard basis in $\mathbb{R}^n$. Without loss of generality we may assume that $$c\epsilon<\min_{1\leq{i}\leq{n}}\{b_{\phi_i}\}\quad\textnormal{and}\quad\eta<\frac{1}{\max_{1\leq i\leq{n-1}}\{\phi_i(c\epsilon)\}}.$$ Clearly, we have
	\begin{equation*}
	x_{m_k}(t)<x(t)-\epsilon\chi_{A_k}(t)\quad\textnormal{a.e.}
	\end{equation*}
	Hence, since $I_{\phi_i}$ is left-continuous, monotone and superadditive for any $i\in\{1,\dots,n-1\}$, by Proposition \ref{prop:modul:mono} we obtain
	\begin{align*}
	\norm{x_{m_k}}{f}{}&\leq\norm{x-\epsilon\chi_{A_k}}{f}{}\leq{1}-\delta_f(\rho(c\epsilon\chi_{A_k}))\\
	&\leq{1}-\delta_f\left(\max_{1\leq i\leq{n-1}}\int_I\phi_{i}\left(c\epsilon\chi_{A_k}\right)d\mu\right)\\
	&\leq{1}-\delta_f\left(\max_{1\leq i\leq{n-1}}\{\phi_{i}(c\epsilon)\}\mu{(A_k)}\right)\\
	&\leq{1}-\delta_f\left(\eta\max_{1\leq i\leq{n-1}}\{\phi_{i}(c\epsilon)\}\right)
	\end{align*}
 	for any $k\in\mathbb{N}$. Since $a_{\phi_j}=0$ for some $j\in\{1,\dots,n-1\}$ we get $$\eta\max_{1\leq{}i\leq{n-1}}\{\phi_{i}(c\epsilon)\}>0.$$ Therefore, passing to subsequence and relabeling if necessary we may easily observe that $\norm{x_m}{f}{}\nrightarrow{1}$. So, this concludes a contradiction and completes the proof. 	
\end{proof}

\begin{theorem}
	Let $f$ be a norm on $\mathbb{R}^n$ that is uniformly monotone. Let $\phi_i$ be an Orlicz function and $\rho_i(x)=I_{\phi_i}(|x|)$ for any $i\in\{1,\dots,n-1\}$ and $x\in{X_\rho}$, where $\rho=\max_{1\leq i\leq{n-1}}\{\rho_i\}$. If $\phi=\max_{1\leq i\leq{n-1}}\{\phi_i\}$ satisfies $\Delta_2$ condition, then $X_\rho$ equipped with the norm $\norm{\cdot}{f}{}$ is lower locally uniformly monotone.
\end{theorem}

\begin{proof}
	Let $(x_m)\subset ({X_\rho})^+$, $x\in({X_\rho})^+$ be such that $x_m\leq{x}$ and $\norm{x_m}{f}{}\rightarrow\norm{x}{f}{}={1}$. Let $\lambda>0$. Since $\phi$ satisfies $\Delta_2$ condition, it is easy to see that $a_\phi=0$, and consequently by definition of $\phi$ there is $j\in\{1,\dots,n-1\}$ such that $a_{\phi_j}=0$. Hence, since $0\leq\lambda(x-x_m)\leq\lambda{x}$ a.e. for any $m\in\mathbb{N}$, by Lemma \ref{lem:conv:in:measu} it follows that $\lambda{x_m}$ converges to $\lambda{x}$ in globally measure. Next, by assumption that $\phi$ satisfies $\Delta_2$ condition, by Proposition \ref{prop:OC:Orlicz} we have $\rho(\lambda{x})<\infty$. In consequence, by Lebesgue Dominated Convergence Theorem (see \cite{Royd}) we infer that $\rho(\lambda(x-x_m))\rightarrow{0}$. Hence, since $\lambda>0$ is arbitrary, we conclude $\norm{x-x_m}{f}{}\rightarrow{0}$ and finish the proof.
\end{proof}

\begin{theorem}\label{thm:ULUM=>OC}
	Let $f$ be a norm on $\mathbb{R}^n$ satisfying \eqref{crucial}. Let $\phi_i$ be an Orlicz function and $\rho_i(x)=I_{\phi_i}(|x|)$ for any $i\in\{1,\dots,n-1\}$ and $x\in{X_\rho}$, where $\rho=\max_{1\leq i\leq{n-1}}\{\rho_i\}$. If the space $X_\rho$ equipped with the norm $\norm{\cdot}{f}{}$ is upper locally uniformly monotone, then the norm $\norm{\cdot}{f}{}$ is order continuous.
\end{theorem}

\begin{proof}
	Suppose for a contrary that the norm $\norm{\cdot}{f}{}$ is not order continuous. Then, there exist $\epsilon>0$ and a sequence $(z_m)\subset(X_\rho)^+$ such that $z_m\downarrow{0}$ a.e. and $\norm{z_m}{f}{}\geq\epsilon$ for all $m\in\mathbb{N}$. Without loss of generality we may assume that $I_{\phi_i}(z_m)<\infty$ for any $m\in\mathbb{N}$ and $i\in\{1,\dots,n-1\}$. Hence, by Lebesgue Dominated Convergence Theorem (see \cite{Royd}) it follows that $I_{\phi_i}(z_m)\rightarrow{0}$ for every $i\in\{1,\dots,n-1\}$. Next, using the diagonal method and passing to subsequence and relabeling if necessary we may assume 
	\begin{equation*}
	I_{\phi_i}(z_m)<\frac{1}{2^m}
	\end{equation*}
	for any ${m\in\mathbb{N}}$ and $i\in\{1,\dots,n-1\}$. Additionally, we may suppose without loss of generality that 
	\begin{equation*}
	\mu\left(I\setminus\bigcup_{m\in\mathbb{N}}\supp(z_m)\right)>0.
	\end{equation*}
	Now, let $\epsilon_1>0$ and $y\in(X_\rho)^+$ be such that $\norm{y}{f}{}=1$, $\supp(y)\subset I\setminus\bigcup_{m\in\mathbb{N}}\supp(z_m)$. Then, there is $k_0>0$ such that 
	\begin{equation}\label{equ:ULUM}
	k_0f\left(e_1+\sum_{k=2}^{n}I_{\phi_{k-1}}\left(\frac{y}{k_0}\right)e_k\right)\leq\norm{y}{f}{}+\epsilon_1.
	\end{equation}
	Define $x_m=k_0z_m$ and $y_m=y+x_m$ for any $m\in\mathbb{N}$. Notice that
	\begin{equation*}
	I_{\phi_i}\left(\frac{x_m}{k_0}\right)=I_{\phi_i}(z_m)<\frac{1}{2^m}
	\end{equation*}
	for all $m\in\mathbb{N}$ and $i\in\{1,\dots,n-1\}$. So, $(x_m)\subset(X_\rho)^+$ and also
	\begin{equation}\label{equ2:ULUM}
	\norm{y_m-y}{f}{}=\norm{x_m}{f}{}=\frac{1}{k_0}\norm{z_m}{f}{}\geq\epsilon{k_0}
	\end{equation}
	for all $m\in\mathbb{N}$. In consequence, we have
	\begin{align*}
	\norm{y_m}{f}{}&\leq 	k_0f\left(e_1+\sum_{k=2}^{n}I_{\phi_{k-1}}\left(\frac{y_m}{k_0}\right)e_k\right)\\
	&=k_0f\left(e_1+\sum_{k=2}^{n}I_{\phi_{k-1}}\left(\frac{y}{k_0}\right)e_k+I_{\phi_{k-1}}\left(\frac{x_m}{k_0}\right)e_k\right)\\
	&\leq{}k_0f\left(e_1+\sum_{k=2}^{n}I_{\phi_{k-1}}\left(\frac{y}{k_0}\right)e_k+\frac{1}{2^m}e_k\right)
	\end{align*}
	for every $m\in\mathbb{N}$. Therefore, by assumption that $f$ is a norm on $\mathbb{R}^n$ we conclude
	\begin{equation*}
	\limsup_{m\rightarrow\infty}\norm{y_m}{f}{}\leq k_0f\left(e_1+\sum_{k=2}^{n}I_{\phi_{k-1}}\left(\frac{y}{k_0}\right)e_k\right).
	\end{equation*}
	Hence, by \eqref{equ:ULUM} we obtain
	\begin{equation*}
	\limsup_{m\rightarrow\infty}\norm{y_m}{f}{}\leq\norm{y}{f}{}+\epsilon_1=1+\epsilon_1.
	\end{equation*}
	Thus, since $\epsilon_1>0$ is arbitrary we get
	\begin{equation*}
	\limsup_{m\rightarrow\infty}\norm{y_m}{f}{}\leq 1,
	\end{equation*}
	whence, since $y_m\geq{y}$ for any $m\in\mathbb{N}$, by assumption that $f$ satisfies \eqref{crucial} it is easy to see $\norm{y_m}{f}{}\geq\norm{y}{f}{}=1$, and consequently $\norm{y_m}{f}{}\rightarrow 1.$ Finally, in view of assumption that $X_\rho$ is upper locally uniformly monotone and by \eqref{equ2:ULUM} we get a contradiction.
\end{proof}

\begin{theorem}
	Let $f$ be a norm on $\mathbb{R}^n$ that is uniformly monotone. Let $\phi_i$ be an Orlicz function and $\rho_i(x)=I_{\phi_i}(|x|)$ for any $i\in\{1,\dots,n-1\}$ and $x\in{X_\rho}$, where $\rho=\max_{1\leq i\leq{n-1}}\{\rho_i\}$. Consider a space $X_\rho$ equipped with a norm $\norm{\cdot}{f}{}$. The following conditions are equivalent.
	\begin{itemize}
		\item[$(i)$] $X_\rho$ is order continuous.
     	\item[$(ii)$] $\phi=\max_{1\leq{}i\leq{n-1}}\{\phi_i\}$ satisfies $\Delta_2$ condition.
		\item[$(iii)$] $X_\rho$ is uniformly monotone.
		\item[$(iv)$]  $X_\rho$ is upper locally uniformly monotone.
		\item[$(v)$]  $X_\rho$ is decreasing uniformly monotone.
		\item[$(vi)$]  $X_\rho$ is lower locally uniformly monotone.
		\item[$(vii)$]  $X_\rho$ is increasing uniformly monotone.
	\end{itemize}
\end{theorem}

\begin{proof}
	Immediately, by Proposition 2.1 in \cite{DoHuLoMaSi} and by Theorem 1.2 in \cite{CheHeHudz} it follows that $(vi)\Leftrightarrow(vii)$. Next, by Theorem 1.1 in \cite{CheHeHudz} and by Theorem \ref{thm:ULUM=>OC} we conclude that $(iv)\Leftrightarrow(v)$. Clearly, we have $(iii)\Rightarrow(iv)$ and $(iii)\Rightarrow(vi)$ (see \cite{CheHeHudz}). Moreover, by Theorem \ref{thm:ULUM=>OC} we infer that $(iv)\Rightarrow(i)$. Similarly, by Proposition 2.1 in \cite{DoHuLoMaSi} we get $(vi)\Rightarrow(i)$. Now, since $f$ is a norm on $\mathbb{R}^n$ that is uniformly monotone we easily observe that $f$ satisfies \eqref{crucial}. Hence, by Proposition \ref{prop:OC:Orlicz} we obtain $(i)\Rightarrow(ii)$. Next, since $\phi$ satisfies $\Delta_2$ condition, by Theorem 8.14 in \cite{Mus} and by Remark \ref{rem:rho-converg} we conclude that $\rho$ satisfies \eqref{rho-converg}. Finally, by Theorem \ref{thm:UM} we infer $(ii)\Rightarrow(iii)$ and finish the proof.
\end{proof}

\subsection*{Acknowledgments}
\begin{flushleft}
	$^1$ This research is supported by the grant 2017/01/X/ST1/01036 from National Science Centre, Poland.	
\end{flushleft}

$\begin{array}{lr}
\textnormal{\small Maciej CIESIELSKI} & \textnormal{\small Grzegorz Lewicki}\\
\textnormal{\small Institute of Mathematics} & \textnormal{\small Department of Mathematics and Computer Science}\\
\textnormal{\small Pozna\'{n} University of Technology} & \textnormal{\small Jagiellonian University}\\
\textnormal{\small Piotrowo 3A, 60-965 Pozna\'{n}, Poland} & \textnormal{\small \L ojasiewicza 6, 30-348 Krak\'ow, Poland}\\
\textnormal{\small email: maciej.ciesielski@put.poznan.pl;} & \textnormal{\small email: grzegorz.lewicki@im.uj.edu.pl}
\end{array}$

\end{document}